\documentclass[a4paper,oneside]{scrartcl}%
\usepackage[utf8]{inputenc}%
\usepackage{amsthm,amsmath,amsfonts,amssymb}%
\usepackage[initials,alphabetic]{amsrefs}
\usepackage{mathrsfs}
\usepackage{url}
\numberwithin{equation}{section}
\allowdisplaybreaks
\usepackage{enumitem}
\usepackage{tikz-cd}%
\usepackage{mathtools}
\mathtoolsset{showonlyrefs=true}
\usepackage{hyperref}
\hypersetup{bookmarksdepth=3,
            colorlinks = true,
            citecolor = {blue},
            linkcolor = {red}
}%
\usepackage[all]{hypcap}
\usepackage{color,soul}%
\usepackage[baseline]{euflag}
\numberwithin{equation}{section}
\DeclareFontFamily{U}{mathx}{\hyphenchar\font45}
\DeclareFontShape{U}{mathx}{m}{n}{
      <5> <6> <7> <8> <9> <10>
      <10.95> <12> <14.4> <17.28> <20.74> <24.88>
      mathx10
      }{}
\DeclareSymbolFont{mathx}{U}{mathx}{m}{n}
\DeclareFontSubstitution{U}{mathx}{m}{n}
\DeclareMathAccent{\widecheck}{0}{mathx}{"71}
\newcommand{\m}[1]{\mathcal{#1}}%
\newcommand{\RR}{\mathbb{R}}
\newcommand{\Alt}{\mathcal{A}}
\DeclareMathOperator{\End}{End}
\newcommand{\dual}{\vee}
\newcommand{\dd}{\mathop{}\!\mathrm{d}}%
\newcommand{\dc}{\mathop{}\!\mathrm{d}^c}%
\newcommand{\ddc}{\mathop{}\!\mathrm{d}\mathrm{d}^c}%
\newcommand{\nablaminus}{\nabla^-\!\dd}
\DeclareMathOperator{\Hess}{\mathrm{Hess}}
\DeclareMathOperator{\I}{\mathrm{i}}
\newcommand{\e}{\mathrm{e}}
\newcommand{\vol}{\dd vol}
\DeclareMathOperator{\Lie}{\text{Lie}}
\newcommand{\contract}{\lrcorner}
\DeclareMathOperator{\chern}{\mathrm{c}}
\DeclareMathOperator{\Ric}{\mathrm{Ric}}
\DeclareMathOperator{\Scal}{\mathrm{Scal}}
\DeclareMathOperator{\ScalTW}{\mathrm{Scal}^{\mathrm{TW}}}
\newcommand{\SasFut}{\mathbf{SF}}
\newcommand{\FS}{\mathrm{FS}}
\newcommand{\Action}{\mathcal{A}}
\newcommand{\actvol}{\mathcal{B}}
\newcommand{\Scaltot}{\mathbf{Scal}}
\newcommand{\Voltot}{\mathbf{Vol}}
\newcommand{\Torus}{\mathbb{T}}
\newcommand{\TorusL}{\mathbb{T}}
\newcommand{\TorusTest}{\hat{\mathbb{T}}}
\newcommand{\TorusX}{\check{\mathbb{T}}}
\newcommand{\torusL}{\mathfrak{t}}
\newcommand{\torusX}{\check{\mathfrak{t}}}
\newcommand{\Rcone}{\torusL_+}
\newcommand{\EH}{\mathrm{ EH}}
\newcommand{\EHT}{\EH^\TorusL}
\newcommand{\REH}{\mathcal{EH}}
\newcommand{\REHmin}{\REH_{\min}}
\newcommand{\cscTWsetT}{\mathcal{C}_{TW}^{\TorusL}}
\newcommand{\Kmet}{\mathcal{P}}
\newcommand{\Cmet}{\mathcal{Z}}
\newcommand{\Sasmap}{\varpi}
\newcommand{\Kahmap}{\pi_*^{\dd}}
\DeclarePairedDelimiter\abs{\lvert}{\rvert}
\DeclarePairedDelimiter\norm{\lVert}{\rVert}
\DeclarePairedDelimiter{\set}{\{}{\}}
\newcommand{\suchthat}{\mathrel{}\mathclose{}\middle|\mathopen{}\mathrel{}}

\newcommand{\tstL}{{\mathscr L}}
\newcommand{\tstN}{{\mathscr N}}

\newcommand{\tstX}{{\mathscr X}}
\newcommand{\TCmap}{{\nu}}

\DeclareMathOperator{\Id}{Id}

\DeclareMathOperator{\Aut}{Aut}

\renewcommand{\phi}{\varphi}

\newcommand{\N}{\mathbb{N}}
\newcommand{\R}{\mathbb{R}}
\newcommand{\C}{\mathbb{C}}

\newcommand{\Z}{\mathbb{Z}}
\newcommand{\Q}{\mathbb{Q}}

\newcommand{\cp}{\mathbb{P}}
\renewcommand{\epsilon}{\varepsilon}

\newcommand{\Pol}{\mathrm{P}}
\newcommand{\F}{\mathbf{f}}
\def\mO{\mathcal{O}}

\def\mL{\mathcal{L}}
\def\bT{\mathbb T}

\def\kt{\mathfrak{t}} 

\def\Ds{\mathrm{D}}

\def\bfV{{\mathbf{V}}}
\def\bfS{{\mathbf{S}}}

\def\bT{\mathbb{T}}

\def\uI{\underline{I}}

\def\proj{\mathrm{p}}
\newcommand{\PSH}{\mathrm{PSH}}
\renewcommand{\ll}{<\!\!<}
\newcommand{\CRYam}{\mathcal{Y}_{CR}}
\newcommand{\CRYamT}{\CRYam^{\TorusL}}
\newcommand{\CRYamInv}{\mathcal{Y}_{\text{sup}}}
\newcommand{\CRYamInvT}{\mathcal{Y}_{\text{sup}}^{\TorusL}}
\newtheorem{proposition}{Proposition}[section]
\newtheorem{lemma}[proposition]{Lemma}
\newtheorem{theorem}[proposition]{Theorem}
\newtheorem{corollary}[proposition]{Corollary}

\theoremstyle{definition}
\newtheorem{definition}[proposition]{Definition}
\newtheorem{remark}[proposition]{Remark}

\title{The CR Yamabe invariant and constant scalar curvature Sasaki metrics}
\author{
    Abdellah Lahdili\thanks{Centre interuniversitaire de recherches en géométrie et topologie, Université du Quèbec à Montréal, 201 avenue du Président Kennedy, Montréal. lahdili.abdellah@gmail.com},
    Eveline Legendre\thanks{Institut Camille Jordan, Université Claude Bernard Lyon 1, 21 av. Claude Bernard, 69100 Villeurbanne, France. legendre@math.univ-lyon1.fr, scarpa@math.univ-lyon1.fr},
    Carlo Scarpa\footnotemark[2]
}
\date{
\vspace{-1cm}}

\begin{document}

\maketitle

\begin{abstract}
    \noindent \textbf{Abstract.} We propose a new approach to the existence of constant transversal scalar curvature Sasaki structures drawing on ideas and tools from the CR Yamabe problem, establishing a link between the CR Yamabe invariant, the existence of Sasaki structures of constant transversal scalar curvature, and the K-stability of Sasaki manifolds. Assuming that the Sasaki-Reeb cone contains a regular vector field, we show that if the CR Yamabe invariant of a compact Sasaki manifold attains a specific value determined by the geometry of the Reeb cone, then the Sasaki manifold is K-semistable. Under the additional assumption of non-positive average scalar curvature, the CR Yamabe invariant attains this topological value if the manifold admits approximately constant scalar curvature Sasaki structures, and we also show a partial converse. As an application, we provide a new numerical criterion for the K-semistability of polarised compact complex manifolds.
\end{abstract}

\setcounter{tocdepth}{2}
\tableofcontents

\section{Introduction}

In this paper, which is a continuation of our previous work~\cite{LahdiliLegendreScarpa}, we study the \emph{CR Yamabe energy} of a polarized compact complex manifold~$(X,L)$. Our motivation is to develop this functional as a tool to study the existence of \emph{constant scalar curvature Sasaki} (cscS, for brevity) structures on the transversally holomorphic circle bundle associated to~$(X,L)$. Up to natural identifications, these cscS structures encompass the possible scalar flat K\"ahler conical structures on the cone associated to~$(X,L)$ and the constant scalar curvature K\"ahler (cscK) metrics in~$\chern_1(L)$.

An important invariant of a Sasaki structure is its Reeb vector field. The set of all possible Reeb vector fields forms the \emph{Sasaki-Reeb cone}~$\Rcone$, a finite-dimensional polyhedral convex cone which plays a role similar to the K\"ahler cone in K\"ahler geometry. A classical obstruction to the existence of a cscS structure with a prescribed Reeb vector field~$\chi\in\Rcone$ is the vanishing of the \emph{transversal Futaki invariant}~\cite{FutakiOnoWang}, which was shown to coincide with the derivative of the \emph{Einstein-Hilbert} functional~$\REH \colon \kt_+ \to\R$ in~\cite{EinsteinHilbert-SasakiFutaki} and can be reduced to the Volume functional when~$X$ is Fano and~$L=-\lambda\,K_X$ for some~$\lambda\in\Q_{>0}$ thanks to~\cite{MartelliSparksYau}. Following~\cite{ApostolovCalderbank_weighted}, any Reeb vector field~$\chi \in \Rcone$ can be identified with a \emph{weight} on~$(X,L)$ in the sense of~\cite{Lahdili_weighted}. Once this weight is fixed, the corresponding \emph{weighted K-energy functional}~$\mathcal{M}_\chi$, defined on the space of K\"ahler metrics in~$c_1(L)$, is able to detect some cscS structure on the associated circle, but only those whose associated Reeb vector field is precisely~$\chi$.

In the present paper we take an alternative approach to the existence of cscS structures. We introduce the \emph{CR Yamabe energy functional} on the set~$\chern_1(L)_+$ of K\"ahler forms in~$\chern_1(L)$,
\begin{equation}
    \CRYam: c_1(L)_+\to\R
\end{equation}
and show that its critical points corresponds to cscS structures, those achieving the minimum of~$\REH$, on the circle bundle associated to~$(X,L)$. 

As a sharp contrast to the aforementioned results on the weighted K-energy, the CR Yamabe energy detects the cscS structures without having to prescribe their Sasaki-Reeb field beforehand. This is a promising advantage to attack open problems regarding the neighborhood of a cscS structure (and eventually the moduli space of such structures), as we illustrate here this by giving a partial solution to the isolation problem of cscS structures (c.f.\ \cite{Sasaki_problems}), see Theorem~\ref{thm:isolation} below.

We now describe our results with more details.

\subsection{Regularity and critical points of the Yamabe energy}
Let~$(X,L)$ be a polarised manifold, and denote by~$\xi_0$ the generator of the vector bundle~$U(1)$-action on~$L$ or on its dual~$L^\dual$. The associated circle bundle is the manifold
\begin{equation}
    N\coloneqq(L^\dual\setminus X)/\R_{>0}
\end{equation}
which comes with a transversal holomorphic structure~$\uI \in \End(TN/\langle\xi_0\rangle)$ corresponding to the complex structure on~$X$. For any K\"ahler form~$\omega\in\chern_1(L)$, the Boothby-Wang construction~\cite{BoothbyWang}, see Section~\ref{sec:CRgeom}, associates a contact form~$\eta_\omega$ with Reeb vector field~$\xi_0$, determined up to isotopy by~$\dd\eta_\omega=\pi^*\omega$. In particular,~$d\eta_\omega$ is of type~$(1,1)$ with respect to~$\uI$ and is \emph{Sasaki}, so that it lies in the set
\begin{equation}
    \Kmet(X,L)\coloneqq \set*{ \eta\in\Alt^1(N,\R) \suchthat \eta\text{ Sasaki with Sasaki-Reeb field }\xi_0 \mbox{ and } \uI\mbox{--compatible}}.
\end{equation}
Any form~$\eta\in \Kmet(X,L)$ defines a conformal class~$[\eta] \coloneqq \{f^{-1}\eta\,|\, f\in C^{\infty}(N,\R_+)^{\xi_0}\}$ of~$\xi_0$-invariant \emph{pseudo-hermitian} structures, which are all compatible with the same CR structure~$(D_\eta= \ker\eta, \uI_\eta\in \End(D_\eta))$. It turns out that the union of these conformal classes
\begin{equation}
    \Cmet(X,L) \coloneqq \bigcup_{\eta\in \Kmet(X,L)} [\eta]  
\end{equation}
is a Frechet manifold, the total space of a fibration~$\Sasmap\colon \Cmet(X,L) \to \Kmet(X,L)$.

\smallskip

The \emph{CR Yamabe problem}, introduced and partially solved by Jerison and Lee~\cite{JerisonLee_CRYamabe}, consists in finding a pseudo-hermitian form~$\alpha$ within a fixed conformal class~$[\eta]$ whose associated \emph{Tanaka-Webster scalar curvature}~$\ScalTW(\alpha)$ is constant (a cscTW structure, for short). These classical notions are recalled in \S\ref{sec:CRgeom}. A key point for our goal is that when~$\alpha$ is Sasaki then~$\ScalTW(\alpha)+2n$ is the scalar curvature of the associated Riemannian metric and, by~\cite{Webster_kahler}, if~$\alpha =\eta_\omega\in \Kmet(X,L)$ then~$\ScalTW(\alpha) = \pi^*\Scal(\omega)$.

Much of the present work is based on the \emph{solution} of the CR Yamabe problem: any conformal class~$[\eta]$ contains a minimiser of the \emph{CR Einstein-Hilbert functional}
\begin{equation}
    \EH(\alpha) \coloneqq \frac{\int_N \ScalTW(\alpha) \alpha\wedge (\dd\alpha)^{n}}{\left(\int_N \alpha\wedge (\dd\alpha)^{n}\right)^{n/n+1}},
\end{equation}
and any such minimiser is a cscTW structure, see~\cites{JerisonLee_CRYamabe, Gamara, Gamara_Yacoub} and Theorem~\ref{thm:JerisonLee_CRYam} below. In particular, the \emph{CR Yamabe energy} functional
\begin{equation}
    \CRYam(\eta) \coloneqq \inf_{f\in C^{\infty}(N,\R_+)} \EH(f^{-1}\eta)    
\end{equation}
is well defined. As we explain in Section~\ref{sec:CRgeom}, we can equally consider the Yamabe energy as a functional on the space of K\"ahler metrics in~$c_1(L)$ via the Boothby-Wang construction, setting~$\CRYam(\omega)\coloneqq \CRYam(\eta_\omega)$. From this point of view, the CR Yamabe energy shares several properties with Perelman's $\mu$-entropy functional as studied by Inoue \cite{Inoue_PerelmanEntropy}. We also refer to \cite{LahdiliLegendreScarpa} for an account of the differences between the CR Einstein-Hilbert functional and Perelman's W-functional considered in \cite{Inoue_PerelmanEntropy}.

It might not be the case in general that~$\CRYam$ is smooth on~$\chern_1(L)_+$, but following ideas in~\cite{Anderson_Yamabe_differentiable} we can show that it is differentiable on some points of~$\Kmet(X,L)$ (conjecturally, almost everywhere).

\begin{theorem}\label{thm:differentiabilityVersionVague} Assume that~$\CRYam: c_1(L)_+ \to \R$ is differentiable at~$\omega$ and that~$\omega$ is a critical point of~$\CRYam$. Then any minimiser~$\alpha \in [\eta_\omega]$ of~$\EH$ is a Sasaki structure of constant scalar curvature. 
\end{theorem}
A more precise statement of this claim is given in Theorem~\ref{thm:CRYam_diff_text}, in which we also provide specific conditions under which~$\CRYam$ is differentiable at a point, as well as an explicit formula for the derivative of~$\CRYam$. We remark that Theorem~\ref{thm:CRYam_diff_text} also parallels the differentiability of Perelman's entropy, see~\cite[equation~$(3.6)$]{TianZhu_KahlerRicci_Fano}.

\smallskip

Note that if~$\alpha\in\Cmet(X,L)$ is a Sasaki structure, the Sasaki-Reeb vector field of~$\alpha$ generates a holomorphic action on~$N$, lifting a holomorphic action on~$X$. This action might be just the fiberwise circle action on~$N$ (e.g.\ if~$\alpha\in\Kmet(X,L)$) but in general it might be quite different. Note also that this action commutes with the fibrewise circle action, as every element of~$\Cmet(X,L)$ is~$\xi_0$-invariant. This motivates the introduction of an \emph{equivariant} version of the CR Yamabe problem. Hence, we fix a compact torus~$\TorusL$ of automorphisms of~$L\to X$ that contains the fiberwise~$U(1)$-action of~$L$, and consider the the restriction of~$\EH$ on the space of~$\TorusL$-invariant forms in~$\Cmet(X,L)$. We will often use the following notation: for any set~$A$ on which~$\bT$ acts, and any function~$F$ defined on~$A$,~$A^\bT$ is the subset of~$\bT$-invariant elements and~$F^\bT$ is the restriction of~$F$ to~$A^\bT$.

We will exhibit an upper bound for the equivariant CR Yamabe energy, relative to the torus~$\bT$. A priori, such a torus can be chosen arbitrarily, but the upper bound we get depends on that choice (the larger the torus, the finer the bound). To define this bound, recall (see also \S\ref{ss:ReebMoment}) that the \emph{Sasaki-Reeb cone}~$\Rcone$ of~$\TorusL$ is the set of vectors of~$\torusL\coloneqq\mbox{Lie} \bT$ that are Reeb vector fields of an element of~$\Cmet(X,L)^\bT$. Given~$\chi\in\Rcone$ and any~$\eta\in\Kmet(X,L)^{\TorusL}$, it is known~\cite{FutakiOnoWang} that
\begin{equation}
    \REH(\chi) \coloneqq \EH(\eta(\chi)^{-1}\eta)
\end{equation}
only depends on~$\chi$, rather than~$\eta$. Moreover, from~\cite{BoyerHuangLegendre_DH},~$\REH$ always achieves a minimum on~$\Rcone$, which we denote by~$\REHmin$. Then, by definition, for any~$\eta\in\Kmet(X,L)^{\TorusL}$ we have
\begin{equation}\label{eq:Tinv_ineq}
    \CRYamT(\eta) \leq \REHmin.
\end{equation}
The following gives a criterion for the existence of a cscS structure on~$N$. It should be compared with the criterion \cite[Theorem $1.4$]{Inoue_PerelmanEntropy} for the existence of $\mu^\lambda$-cscK metrics.
\begin{theorem}\label{thm:CRYam_equal_cscTW}
    If~$\eta\in\Kmet(X,L)^{\TorusL}$ realises the equality in~\eqref{eq:Tinv_ineq}, then~$\eta(\chi)^{-1}\eta$ is cscS for any minimiser~$\chi\in\Rcone$ of~$\REH$. The converse holds, when~$\REHmin\leq 0$.
\end{theorem}
This shows that the CR Yamabe energy identifies the cscS structures whose Sasaki-Reeb vector fields minimise~$\REH$. In contrast, we proved in~\cite{LahdiliLegendreScarpa} that every cscS structure is a critical point of the Einstein-Hilbert functional, giving a way to study the space of cscS structures on~$(X,L)$. As an example, we prove the following result on the \emph{isolation} of cscS structures.
\begin{theorem}\label{thm:isolation}
    Assume that~$\eta\in\Cmet(X,L)$ is a Sasaki form of constant TW-scalar curvature~$\ScalTW(\eta)\eqqcolon c_\eta$. Let~$\bT$ be a maximal compact torus of automorphisms of~$(X,L)$ containing~$R^\eta$, and denote by~$\lambda_1^{\TorusL}(\eta)$ the first non-zero eigenvalue of of the basic Laplacian~$\Delta_\eta$ on the space of~$\bT$-invariant functions. If
    \begin{equation}\label{eq:isol_condition}
        \lambda_1^{\TorusL}(\eta)>\frac{c_\eta}{2(n+1)}
    \end{equation}
    then every cscS structure in a neighbourhood of~$\eta$ in~$\Cmet(X,L)^{\TorusL}$ is in the~$\TorusL^{\C}$-orbit of~$\eta$.
\end{theorem}
Condition~\eqref{eq:isol_condition} holds e.g.\ if~$c_\eta\leq 0$ (as~$\lambda_1^{\TorusL}(\eta)>0$), or if~$\eta$ is Sasaki-$\eta$-Einstein of positive curvature, see~\cite[Theorem~$1.7$]{EinsteinHilbert-SasakiFutaki}.

\subsection{The CR Yamabe invariant and K-semistability}

Given~$(X,L,\TorusL)$ as above, we defined their \emph{CR Yamabe invariant} as
\begin{equation}
    \CRYamInv^{\TorusL}(X,L)\coloneqq \sup\set{\CRYamT(\eta)\mid\eta\in\Kmet(X,L)^{\TorusL}}.
\end{equation}
and note that~$\CRYamInv^{\TorusL}(X,L)\leq\REHmin$ by~\eqref{eq:Tinv_ineq}. We aim to show that the condition
\begin{equation}
    \CRYamInv^{\TorusL}(X,L)=\REHmin    
\end{equation}
is closely related to the existence of cscS structures and the Sasaki K-stability of~\cite{CollinsSzekelyhidi_KsemistabSasaki}. As a first step in this direction, we establish the following result, which can be seen as an algebraic counterpart to Theorem~\ref{thm:CRYam_equal_cscTW}.
\begin{theorem}\label{thm:Ksemistable}
    Let~$\chi_{\min}\in\Rcone$ be a minimiser of~$\REH$. If~$\CRYamInv^{\TorusL}(X,L)=\REHmin$, then the Sasaki manifold~$(N,\uI,\chi_{\min})$ is K-semistable.
\end{theorem}
This is in fact a consequence of a refinement of the inequality~$\CRYamInv^{\TorusL}(X,L)\leq\REHmin$, that we establish as a slight generalisation of the results in~\cite{LahdiliLegendreScarpa} in Section~\ref{sec:Kstab}. Briefly, we consider a smooth, dominant, ample~$\TorusL$-equivariant test configuration~$(\tstX,\tstL)$ for~$(X,L)$, with reduced central fibre. For any~$\chi\in\Rcone$ and sufficiently small~$s>0$, we introduce a numerical invariant~$\EH_s^\chi(\tstX,\tstL)$ which is analytic in~$s$ and can be interpreted as the Einstein-Hilbert functional on the central fibre of~$(\tstX,\tstL)$ evaluated on the Reeb vector field~$\chi-s\zeta$, where~$\zeta$ is the generator of the test configuration~$\C^*$-action. We prove
\begin{theorem}\label{Thm:boundsEHchi}
    For every~$\chi\in\Rcone$ and every test configuration~$(\tstX,\tstL)$ as above,
    \begin{equation}
        \CRYamInv^{\TorusL}(X,L)\leq\EH^{\chi}_s(\tstX,\tstL).
    \end{equation}
\end{theorem}
This result should be compared with \cite[Theorem $1.5$]{Inoue_PerelmanEntropy} for the $\mu$-entropy. The precise definition of~$\EH_s^\chi(\tstX,\tstL)$ is given in Section~\ref{sec:Kstab} together with an extension of the main theorems of~\cite{LahdiliLegendreScarpa}. In particular, we show in Section~\ref{sec:Kstab} (see Lemma~\ref{lem:derivEH=DF}) that~$\partial_{s=0} \EH^{\chi}_s(\tstX,\tstL)$ is the global Sasaki-Futaki invariant of the test configuration defined in~\cite{ApostolovCalderbankLegendre_weighted}, which governs the K-semistability of Sasaki manifolds.

\subsection{Sasaki manifolds of negative total scalar curvature}

When~$\REHmin\leq 0$, we obtain an alternative description of~$\CRYamInv^{\TorusL}(X,L)$ following LeBrun's characterisation~\cite{LeBrun_YamabeKodaira} of the Riemannian Yamabe invariant in the case of non-positive total scalar curvature. This expression for the CR Yamabe invariant had essentially been noticed already in~\cite[Proposition~$6.3$]{Dietrich_CRYam} and~\cite[Theorem~$3.1$]{SungTakeuchi_noneq_CR}, here we present it in a slightly different,~$\TorusL$-invariant form, and highlight some consequences.

\begin{theorem}\label{thm:CRYam_formula}
    If~$\CRYamInv^{\TorusL}(X,L)\leq 0$, e.g.\ if~$c_1(X).c_1(L)^n\leq 0$, then
    \begin{equation}
        \abs{\CRYamInv^{\TorusL}(X,L)}=\inf_{\alpha\in\Cmet(X,L)^{\TorusL}} \norm*{\ScalTW(\alpha)}_{L^{n+1}(\alpha)}.
    \end{equation}
\end{theorem}
It seems likely that this formula can be used to compute the CR Yamabe invariant of Sasaki manifolds of negative average curvature, along the lines of~\cites{Kobayashi_Yamabe,Petean_surgery_Yam,Petean_Yam_computations} for the Riemannian Yamabe invariant, see~\cite{Dietrich_CRYam} for some results in this direction. Note that we can easily create Sasaki manifolds with negative total TW scalar curvature using join products with circle bundles over Riemann surfaces of high genus.

For now, we use this description to relate the equality~$\CRYamInvT(X,L)=\REHmin$ with the existence of \emph{approximate cscS structures} which we define as follows.
\begin{definition}\label{def:approx_cscS}
    Let~$\chi_{\min}\in\Rcone$ be a minimum of~$\REH$, and fix~$p\in[1,+\infty]$. We say that~$(X,L)$ admits \emph{$L^p$-approximate cscS metrics} if for every~$\varepsilon>0$, there exists a Sasakian form~$\alpha_\varepsilon\in\Cmet(X,L)^{\TorusL}$ of Reeb vector field~$\chi_{\min}$ such that
    \begin{equation}
        \norm*{\ScalTW(\alpha_\varepsilon)-\REHmin}_{L^p(\alpha_\varepsilon)}<\varepsilon.
    \end{equation}
\end{definition}
We will deduce from Theorem~\ref{thm:CRYam_formula}
\begin{corollary}\label{cor:approx_cscK}
    Assume that~$\REHmin\leq 0$. If~$(X,L)$ admits~$L^p$-approximate cscS metrics for some~$p\in[n+1,+\infty]$, then~$\CRYamInvT(X,L)=\REHmin$.
\end{corollary}
We will discuss at some length in Section~\ref{sec:negativecurv} the possibility to have a converse to Corollary~\ref{cor:approx_cscK}, which we establish in a weak sense in Section~\ref{sec:weakconverse}.

\subsection{Applications to the cscK problem}

Assume now that~$\REH(\xi_0) = \REHmin$. As every Sasaki form in~$\Kmet(X,L)$ has~$\xi_0$ as Sasaki-Reeb vector field, Theorem~\ref{thm:Ksemistable} and Corollary~\ref{cor:approx_cscK} are concerned with cscS structures in~$\Kmet(X,L)$, which in turn correspond to cscK metrics in~$c_1(L)_+$. Note also that~$\REH(\xi_0)$ is a positive multiple of~$\chern_1(X).\chern_1(L)^{n-1}$, so we can summarise our results as follows
\begin{proposition}\label{prop:abc}
    Assume that~$\REH(\xi_0) = \REHmin$ and~$\chern_1(X).\chern_1(L)^{n-1}\leq 0$. Then~$(a)\Rightarrow(b)\Rightarrow(c)$, for
    \begin{enumerate}[label=(\emph{\alph*})]
        \item~$(X,L)$ has~$L^p$-approximate cscK metrics;
        \item~$\CRYamInv(X,L)=\REHmin$;
        \item~$(X,L)$ is K-semistable.
    \end{enumerate}
\end{proposition}
A simple case when~$\REH(\xi_0) = \REHmin$ holds is if~$\TorusL$ coincides with the fibrewise~$U(1)$-action (e.g.\ if~$(X,L)$ has no automorphisms), as in that case~$\Rcone=\R_+\xi_0$. However the equality might also be true for larger tori, like for the examples revisited in \S\ref{sec:Examples}.

We conjecture that the three conditions of Proposition~\ref{prop:abc} should all be equivalent if~$\chern_1(X).\chern_1(L)^{n-1}\leq 0$, at least in the special case when the reduced automorphism group of~$X$ is trivial. Note also that~$(a)\Rightarrow(c)$ and~$(b)\Rightarrow(c)$ hold even without the negative average curvature assumption, while the implication~$(c)\Rightarrow(a)$ would be a \emph{weak} version of the YTD conjecture. 

The weak YTD conjecture was proven on Fano manifolds (for~$L=-K_X$) in~\cite{ChiLi_weakYTD} by combining and building on the works of~\cite{Donaldson_lower_bounds} and~\cite{Bando_almostKE}. In the Fano case, the CR Yamabe invariant of Proposition~\ref{prop:abc} is replaced by the Perelman energy. Of course, the assumption of negative average curvature places our conjecture well outside the Fano setting, but we hope that our approach to cscK metrics using the CR Einstein-Hilbert functional will be useful to establish an analogue of the weak YTD conjecture for K\"ahler manifolds with non-positive average scalar curvature.

\paragraph{Plan of the paper and some remarks.}
We recall in Section~\ref{sec:CRgeom} some CR geometry and we give an overview of the CR Yamabe problem. In particular, Section~\ref{sec:CRYam} contains a solution of the equivariant CR Yamabe problem. We then show in Section~\ref{sec:Examples} how Theorem~\ref{thm:CRYam_equal_cscTW} follows from the equivariant CR Yamabe problem.

Section~\ref{sec:regularity} contains several regularity results. Theorem~\ref{thm:CRYam_diff_text} gives a criterion for the differentiability of~$\CRYamT$ and an expression for its gradient, from which Theorem~\ref{thm:differentiabilityVersionVague} follows. In Section~\ref{sec:Hessian} we compute the second derivative of~$\EH$ and~$\CRYam$ around a critical point and we deduce our isolation result, Theorem~\ref{thm:isolation}.

In Section~\ref{sec:negativecurv} we prove the alternative characterisation of the CR Yamabe energy of Theorem~\ref{thm:CRYam_formula} and give a weak converse to Corollary~\ref{cor:approx_cscK}.

Finally, Section~\ref{sec:Kstab} is divided in two parts. In \S\ref{sec:EH_testconf} we precisely describe the behaviour of the Einstein-Hilbert functional near the central fibre of a Sasaki test configuration, Theorem~\ref{thm:EH_limit_testconf} in particular, which generalise some results of~\cite{LahdiliLegendreScarpa}. We then show how Theorem~\ref{thm:Ksemistable} and Theorem~\ref{Thm:boundsEHchi} follow from this. Section \S\ref{sec:Appendix} instead is longer and more technical: we show there how to adapt the arguments of~\cite{LahdiliLegendreScarpa} to Sasaki test configurations with arbitrary Reeb vector field.

\begin{remark}
    As was mentioned above, the cscS structures on the circle bundle over~$(X,L)$ can also be regarded as a particular class of \emph{weighted} cscK metrics in~$\chern_1(L)^{\TorusL}_+$, in the sense of~\cite{Lahdili_weighted}. Most of our results can be restated for these weighted cscK metrics without assuming that the K\"ahler class is Hodge.
\end{remark}

\begin{remark}
    Rather than considering the transversally complex manifold~$(N,\uI,\xi_0)$ associated to a polarisation~$L\to X$, we could present our results for an arbitrary Sasaki manifold with a torus action that admits at least one regular Reeb vector field~$\xi_0$, the two points of view are equivalent. Of course, many of the statements make sense without any regularity assumptions. The results in Section~\ref{sec:Kstab} in particular, should hold in broader generality.
\end{remark}

\paragraph*{Acknowledgements.}
We would like to thank Eiji Inoue for a detailed discussion of his paper \cite{Inoue_PerelmanEntropy} and some helpful remarks. We also thank Christina T{\o}nnesen-Friedman for useful comments on an earlier version of this paper.

C.S. is supported by a MSCA Postdoctoral Fellowship, funded by the Research and Innovation framework programme Horizon Europe {\normalsize\euflag} under grant agreement 101149320. E.L. is partially supported by the grant PRCI ANR--FAPESP: ANR-21-CE40-0017. Part of this work was conducted during the Winter School on K-stability at CIRM.

\section{The CR Yamabe invariant and the cscS problem}\label{sec:CRgeom}

In this Section, we review the basics of CR geometry, mostly to fix notation. We refer to~\cite[Section~$2.1$]{LahdiliLegendreScarpa} for a brief recap on CR structures and the pseudo-hermitian theory, and to~\cites{DragomirTomassini,BoyerGalicki} for a more in-depth introduction to the subject. We then give an account of the (equivariant) CR Yamabe problem and its relation to the geometry of the Sasaki-Reeb cone. We use, as much as possible, the same notation of~\cite[\S 2.1]{LahdiliLegendreScarpa} that we recall briefly for the reader's convenience.

Given a (compact)~$(2n+1)$-dimensional smooth manifold~$N$, a CR contact (or \emph{pseudo-hermitian}) structure~$(N,D,\uI,\alpha)$ is given by a~$2n$-dimensional distribution~$D\subset TN$ and
\begin{itemize}
    \item an integrable complex structure~$\uI$ on~$D$, i.e.\ $\uI\in\End(D)$ such that~$\uI^2=-\Id_D$ and~$D^{1,0}\subset T_{\C}N$ is Frobenius integrable;
    \item $\alpha\in\Alt^1(N,\R)$ such that~$D=\ker\alpha$,~$\dd\alpha_{\restriction D}$ is of type~$(1,1)$ with respect to~$\uI$, and~$\dd\alpha(\cdot,\uI\cdot)>0$ on~$D$.
\end{itemize}
We denote by~$R^\alpha$ the \emph{Reeb vector field} of~$\alpha$, the unique vector field on~$N$ that satisfies~$\alpha(R^\alpha)=1$ and~$\mL_{R^\alpha}\alpha=0$. We recall that a pseudo-hermitian structure~$(N,D,\uI,\alpha)$ is said to be \emph{Sasaki} if and only if~$\mL_{R^\alpha}\uI=0$. In that case we say that~$\alpha$ is a \emph{Sasaki contact} form (or simply a \emph{Sasaki form}) for~$(N,D,\uI)$ and~$R^\alpha$ is a \emph{Sasaki}-Reeb vector field. We often use the symbol~$\eta$ to denote a Sasaki contact form and keep~$\alpha$ for a generic CR contact~$1$-form.

\smallskip

A compact Sasaki manifold~$(N,D,\uI,\eta)$ is \emph{regular} if the flow of~$R^\eta$ induces a free~$U(1)$-action on~$N$. In general, a Sasaki-Reeb vector field must lie in the Lie algebra of a compact torus~$\TorusL\subset\textrm{Aut}(N,D,\uI)$ and we fix such a torus, not necessarily maximal. We are interested in~$\TorusL$-invariant objects, in particular the~$\TorusL$-invariant conformal class
\begin{equation}\label{eq:Reeb_param}
    [\eta]^{\TorusL}\coloneqq\set*{f^{-1}\eta \suchthat f\in  \m{C}^\infty(N,\R_{>0})^{\TorusL}}.    
\end{equation}
For each~$\alpha \in [\eta]^{\TorusL}$,~$\ker\alpha=\ker\eta=D$, so~$\alpha$ defines a~${\TorusL}$-invariant CR contact structure~$(D,\uI,\alpha)$. Conversely, any two CR contact forms with the same kernel belong to the same conformal class.

A reason for considering the parametrisation~\eqref{eq:Reeb_param} of the conformal class of a Sasaki form~$\eta$ is that it gives a simple characterisation of the Sasaki forms conformal to~$\eta$.
\begin{lemma}[\S$8.2.3$\cite{BoyerGalicki}]\label{lemma:Killing}
    Let~$\eta'=f^{-1}\eta$ be a CR contact form on~$(N,D,\uI)$. If~$\eta$ is Sasaki, then~$\eta'$ is Sasaki if and only~$f$ is a transversal Killing potential for the transversal K\"ahler structure~$\dd\eta_{\restriction D}$.   
\end{lemma}

\subsection{Revisiting the (equivariant) CR Yamabe problem}\label{sec:CRYam}

The classical theory of~\cites{Tanaka_CR,Webster_connection} associates a Riemannian metric and a connection to each CR contact structure~$(D,\uI,\alpha)$ on~$N$. This \emph{Tanaka-Webster connection} preserves the CR contact structure~$(D,\uI,\alpha)$ and has prescribed torsion.
The Tanaka-Webster connection in turns defines the \emph{Tanaka-Webster scalar curvature} of~$(N,D,\uI,\alpha)$, a smooth function denoted by~$\ScalTW(\alpha)$. The \emph{$\TorusL$-equivariant CR Yamabe problem} on a~$\TorusL$-invariant CR contact manifold~$(N,D,\uI,\alpha)$ consists in finding a form in~$[\alpha]^{\TorusL}$ with constant Tanaka-Webster scalar curvature (cscTW, for brevity).

As mentioned in the Introduction, the ($\TorusL$-invariant) cscTW forms in a class~$[\eta]^{\TorusL}$ are the critical points of the \emph{Einstein-Hilbert functional}~$\EH\colon[\eta]^{\TorusL} \to \R$ defined by
\begin{equation}\label{eq:EH_def}
    \EH(\alpha)\coloneqq\frac{\Scaltot(\alpha)}{\Voltot(\alpha)^{\frac{n}{n+1}}}
\end{equation}
where~$\Scaltot$ and~$\Voltot$ denote the \emph{total TW-scalar curvature} and the \emph{total volume} of a pseudo-hermitian structure,
\begin{equation}
    \Scaltot(\alpha) \coloneqq
    \int_N \ScalTW(\alpha)\,\alpha\wedge \dd \alpha^{[n]}\ 
    \text{ and }\ 
    \Voltot(\alpha) \coloneqq \int_N  \alpha\wedge \dd \alpha^{[n]}.
\end{equation}
The following formula, proved in~\cite{JerisonLee_CRYamabe,Tanno_contact}, shows how~$\ScalTW$ changes under a conformal transformation of the contact form, using the~$R^\alpha$-basic Laplacian~$\Delta_B f\coloneqq\Delta f - R^\alpha(R^\alpha(f))$ and differential~$\dd_B f\coloneqq \dd f- R^\alpha(f)\alpha$ of~$(D,\uI,\alpha)$
\begin{equation}\label{eq:Tanno}
    \ScalTW(f^{-1}\alpha)=f\ScalTW(\alpha)-2(n+1)\Delta_B f-(n+1)(n+2)f^{-1}\abs{\dd_B f}^2.
\end{equation}
In particular, this expression can be used to express the total TW-scalar curvature and volume as
\begin{equation}\label{eq:Scaltot_Reeb}
\begin{split}
    \Scaltot(f^{-1}\eta)=& \int_N f^{-n}\left(\ScalTW(\eta)+n(n+1)\abs*{\dd\log f}^2\right)\eta\wedge(\dd\eta)^{[n]}\\
    \Voltot(f^{-1}\eta)=& \int_N f^{-n-1}\eta\wedge(\dd\eta)^{[n]}.
\end{split}
\end{equation}
For our treatment of the CR Yamabe problem however, it will often be more convenient to use what we call the \emph{Jerison-Lee} parametrisation of~$[\alpha]^{\TorusL}$, which amounts to setting~$u^{-2/n}=f$ in the previous formulas. Then, one finds
\begin{equation}\label{eq:scal_Reeb}
    \ScalTW(u^{2/n}\alpha)u^{q-1}=2q\Delta_B u+u\,\ScalTW(\alpha)
\end{equation}
where~$q\coloneqq 2(n+1)/n$ is the critical exponent for the Folland-Stein spaces on~$N$, see~\cite{JerisonLee_CRYamabe}. Then, the cscTW condition becomes the following PDE for a function~$u>0$,
\begin{equation}\label{eq:cscTW_u}
    2q\Delta_B u+u\,\ScalTW(\alpha)=u^{q-1}\,c.
\end{equation}

The following statement sums up the main features of the equivariant CR Yamabe problem that we will need.
\begin{theorem}[\cite{JerisonLee_CRYamabe,JerisonLee_CRYamabe_local,Gamara_Yacoub,Gamara,Zhang_Kcontact}]\label{thm:JerisonLee_CRYam}
    Given a~$\TorusL$-invariant compact CR contact manifold~$(N,D,\uI,\alpha_0)$, the~$\TorusL$-invariant cscTW forms are the critical points of the \emph{CR Einstein-Hilbert functional}~\eqref{eq:EH_def}, seen as a functional on~$[\alpha_0]^{\TorusL}$.  Moreover, if there exists~$\alpha\in[\alpha_0]^{\TorusL}$ such that~$\EH(\alpha)\leq 0$, there is at most one cscTW form in~$[\alpha_0]^{\TorusL}$.
    
    Assume that~$[\alpha_0]^{\TorusL}$ contains a \emph{regular} Sasaki form~$\eta$ such that~$R^\eta\in\Lie\TorusL$. Then the Einstein-Hilbert functional has a minimum in~$[\alpha_0]^{\TorusL}$, which is in particular cscTW.
\end{theorem}
The cited works prove a non-equivariant version of the statement, for which no Sasakian or regularity assumption is needed. Zhang~\cite{Zhang_Kcontact} showed a circle invariant version in the K-contact case. We want however to take into account arbitrary torus actions on Sasaki manifolds, and the regularity assumption allows us to reduce the problem on a K\"ahler manifold, greatly simplifying the analysis. Theorem~\ref{thm:JerisonLee_CRYam} is a consequence of Proposition~\ref{prop:CRYam_solution} below.  

Note that the~$\EH$-functional~$\EH\colon[\alpha]^{\TorusL}\to\R$ is invariant under constant rescalings of the contact forms. Moreover, it takes a convenient form with respect to the Jerison-Lee parametrisation of~$[\alpha]^{\TorusL}$
\begin{equation}\label{eq:EH_contactJLcoord}
   \EH(u^{2/n}\alpha) =\frac{\int_N\left(u^2\,\ScalTW(\alpha)+2q\abs{\dd_B u}^2_{\alpha}\right)\alpha\wedge (\dd\alpha)^{[n]}}{\norm{u}^2_{L^q(\alpha)}}
\end{equation}
where~$L^q(\alpha)$ is the~$L^q$-norm with respect to the volume form~$\alpha\wedge(\dd\alpha)^{[n]}$.
    
\begin{proposition}\label{prop:CRYam_solution}
    Let~$\eta$ be a~$\TorusL$-invariant Sasaki structure with regular Reeb vector field~$R^\eta\in\Lie(\TorusL)$. Then, the value
    \begin{equation}\label{eq:CRYam_def}
        \CRYamT(\eta)\coloneqq\inf\set*{\EH(u^{2/n}\eta) \suchthat u\in W^{1,2}(N,\eta)^{\TorusL}}
    \end{equation}
    is realised by a smooth positive function~$u\in \m{C}^{\infty}(N,\R_{>0})^{\TorusL}$, which solves the cscTW equation~\eqref{eq:cscTW_u}
    with~$c=\CRYamT(\eta)\norm{u}^{2-q}_{L^q(\eta)}$.
\end{proposition}
We give here a proof of Proposition~\ref{prop:CRYam_solution} that follows the direct method of the calculus of variations, along the lines of~\cite[\S$4$]{LeeParker_Yamabe}, see in particular the proof of Proposition~$4.2$ \emph{ibid.} This will also be convenient for our regularity results in Section~\ref{sec:regularity}. The key observation to prove the existence of a minimum of~$\EH$ in~$[\eta]^{\TorusL}$ is that the assumptions of Proposition~\ref{prop:CRYam_solution} allow us to simply consider the cscTW equation~\eqref{eq:cscTW_u} on the quotient by the~$R^\eta$-action, a K\"ahler manifold on which~\eqref{eq:cscTW_u} becomes \emph{sub}critical.

\begin{proof}[Proof of Proposition~\ref{prop:CRYam_solution}]
    Let~$M$ be the quotient of~$N$ by the~$U(1)$-action of~$R^\eta$. This manifold inherits a complex structure from~$\uI$, and the curvature~$\dd\eta$ induces a K\"ahler form~$\omega$ on~$M$. Then,~$u$ is a critical point of~$u\mapsto\EH(u^{2/n}\eta)$ if and only if, for every~$\dot{u}\in\m{C}^{\infty}(M,\RR)^{\TorusL}$,
    \begin{equation}
        \int_M \dot{u}\left(u\,\Scal(\omega)+2q\Delta_\omega u -u^{q-1}\frac{\int_M\left(u^2\,\Scal(\omega)+2q\abs{\dd u}^2_{\omega}\right)\omega^{[n]}}{\int_M u^q \omega^{[n]}}\right)\omega^{[n]}=0.
    \end{equation}
    This shows that if~$u$ realises~$\CRYamT(\eta)$ then it must satisfy~\eqref{eq:cscTW_u}, with a constant determined by
    \begin{equation}
        c=\frac{\int_M\left(u^2\,\Scal(\omega)+2q\abs{\dd u}^2_{\omega}\right)\omega^{[n]}}{\int_M u^q \omega^{[n]}}= \EH(u^{2/n}\eta)\norm{u}^{2}_{L^q(\eta)}\norm{u}^{-q}_{L^q(\eta)}.
    \end{equation}
    The exponent~$q=2+2/n$ appearing in the cscTW equation~\eqref{eq:cscTW_u} is smaller than the critical Sobolev exponent for the embedding of~$W^{1,2}(M,\omega)$,~$2^*=2+2/(n-1)$. Hence, the embedding~$W^{1,2}\hookrightarrow L^q$ is compact, this holds for~$\TorusL$-invariant functions as well since this is a closed condition, and we can apply the direct method of the calculus of variations. As~$\EH$ is homogeneous, we only consider only the conformal factors~$u$ such that~$\norm{u}_{L^q(\eta)}=1$. Under this hypothesis, the H\"older inequality readily shows that
    \begin{equation}\label{eq:EH_boundedbelow}
        \EH(u^{2/n}\eta)\geq\int_M u^2\,\Scal(\omega)\,\omega^{[n]}\geq-\norm{\Scal(\omega)}_{L^{n+1}(\omega)}.
    \end{equation}
    So~$\EH$ is bounded below on~$[\eta]^{\TorusL}$ and we can consider a sequence~$u_j$ of smooth positive~$\TorusL$-invariant functions that satisfy~$\norm{u_j}_{L^q(\eta)}=1$ and
    \begin{equation}
        \lim_{j\to\infty}\EH(u_j^{2/n}\eta)=\CRYamT(\eta).
    \end{equation}
    We claim that any such sequence is bounded in~$W^{1,2}$. Indeed,
    \begin{equation}
        2q\int_M\left(u^2+\abs{\dd u}^2\right)\omega^{[n]} = \EH(u^{\frac{2}{n}}\eta) +\int_M u^2 \left(2q-\Scal(\omega)\right)\omega^{[n]}
    \end{equation}
    and~$\int_M u^2\left(2q-\Scal(\omega)\right)\omega^{[n]}\leq 2q+2q\norm{\Scal(\omega)}_{L^{n+1}(\omega)}$, using~$\norm{u}_{L^q}=1$ and Hölder's inequality. By the Rellich-Kondrachov Theorem the embedding~$W^{1,2}\hookrightarrow L^q$ is compact, and up to extracting a subsequence there exists a limit~$u=\lim_{j\to\infty}u_j$ of unitary~$L^q$-norm that realises~$\inf_{u\in\m{C}^\infty(M,\R_{>0})^{\TorusL}}\EH$. This limit~$u$ is a weak solution of the cscTW equation~\eqref{eq:cscTW_u}. By elliptic regularity and the maximum principle,~$u$ is in fact smooth, positive and~$\bT$-invariant.
\end{proof}

\subsection{Varying the CR structures over a polarized K\"ahler manifold}\label{sec:torus_CR}

We assume that~$L\xrightarrow{\pi}X$ is an ample line bundle on a compact complex manifold~$(X,J)$, and we denote by~$\xi_0$ the generator of the~$U(1)$-action on the fibres of~$L$. From now on, we will always consider the transversal holomorphic manifold associated to~$(X,L)$ by the following construction.

Boothby and Wang~\cite{BoothbyWang} showed how to associate a regular contact manifold~$(N,\eta_\omega)$ to any symplectic manifold~$(X,\omega)$ satisfying~$[\omega] \in H^2(X,\Z)$, in such a way that~$N$ has the structure of a~$U(1)$-principal bundle over~$X$, and~$\eta_\omega$ is determined (up to isotopy) by~$\pi^*\omega=\dd\eta_\omega$. In our case~$(X,\omega)$ is K\"ahler with~$\omega\in\chern_1(L)$, and the principal bundle can be taken to be    
\begin{equation}\label{eq:contactmanifold}
    N=\left(L^\dual\setminus X\right)/\R_{>0} \xrightarrow{\pi} X.
\end{equation}
If~$\omega$ is the curvature form of a Hermitian metric~$h$ on~$L$, a possible choice of contact form is~$\eta_\omega\coloneqq\dc\log r_h$, where~$r_h$ is the norm induced on~$L^\dual$. In this case,~$N$ also comes with a \emph{transversal holomorphic structure}, i.e.\ an endomorphism~$\uI$ of~$TN/\langle\xi_0\rangle \simeq \pi^*TX$, induced by the complex structure of~$X$, and~$\eta_\omega$ is a CR contact form for the CR manifold~$(N,\ker\eta_\omega,\uI)$. Note that~$\eta_\omega$ is actually a Sasaki form, as its Reeb vector field~$\xi_0$ is real-holomorphic on~$L$.

\smallskip

From now on, we fix a compact torus action~$\TorusL\curvearrowright L$ that contains the vector bundle~$U(1)$-action, inducing an effective~$\TorusL$-action on~$N$. We assume that~$\TorusL\curvearrowright N$ covers an effective action~$\TorusX\curvearrowright X$ of~$\TorusX=\TorusL/U(1)$.
It will often be notationally convenient to consider the non-effective action of~$\TorusL$ on~$X$. For example, we denote by~$\chern_1(L)^{\TorusL}_+$ the space of~$\TorusX$-invariant K\"ahler forms in~$\chern_1(L)$. Note also that a~$\TorusL$-invariant function on~$N$ is in particular basic with respect to~$\xi_0$, so it can be expressed as the pullback of a~$\TorusX$-invariant function on~$X$. We will often identify functions on~$X$ with~$\xi_0$-invariant functions on~$N$ instead of indicating the pullback operation explicitly.

\smallskip

Given such a torus action, we consider the set~$\Kmet(X,L)^{\TorusL}$ of~$\TorusL$-invariant CR contact forms with Reeb vector field~$\xi_0$, which are all Sasaki.
Any~$\eta\in\Kmet(X,L)^{\TorusL}$ determines a \emph{conformal class} of~$\TorusL$-invariant CR contact forms~$[\eta]^{\TorusL}$ as in the previous subsection, that are all compatible with the same CR structure~$(N,\ker \eta,\uI)$. We denote by~$\Cmet(X,L)^{\TorusL}$ the set of all CR contact forms that can be obtained in this way, that is
\begin{equation}
    \Cmet(X,L)^{\TorusL}=\set*{ \alpha\in\Alt^1(N,\R)^{\TorusL} \suchthat \alpha(\xi_0)>0 \text{ and } \alpha(\xi_0)^{-1}\alpha\in\Kmet(X,L)^{\TorusL}}.
\end{equation}
To summarize, we have the following commutative diagram of surjective maps
\begin{equation}\label{eq:proj_diagram}
    \begin{tikzcd}
        \Cmet(X,L)^{\TorusL}
        \arrow[drr, bend left, "\Pi"]
        \arrow[dr, "\Sasmap_{\xi_0}" ] & & \\
        &  \Kmet(X,L)^{\TorusL} \arrow[r, "\Kahmap"] 
        & \chern_1(L)_+^{\TorusL}
    \end{tikzcd}
\end{equation}
where~$\Kahmap(\eta)=\omega$ is defined by~$\pi^*\omega=\dd\eta$, and~$\Sasmap_{\xi_0}(\alpha)= \alpha(\xi_0)^{-1}\alpha$.

\begin{remark}\label{rmk:Kahlsection0}
    Having fixed one choice of primitive~$\eta_\omega \in \Kmet(X,L)^{\TorusL}$ for some~$\omega \in c_1(L)_+^{\TorusL}$, the~$\ddc$-Lemma gives a way of defining a section~$\omega'\mapsto\eta_{\omega'}$ of~$\Kahmap$.
    Explicitly, for~$\omega'=\omega+\ddc\varphi$, we define~$\eta_{\omega'}=\eta_\omega+\pi^*\dc_{\xi_0}\varphi$, where~$\dc_{\xi_0} = (I^{\xi_0})^{-1}\circ \dd\circ I^{\xi_0}$ is the exterior differential operator twisted by the endomorphism~$I^{\xi_0}$ extending~$\uI$ as~$I^{\xi_0}\xi_0=0$.
    
    More generally, any~$\eta',\eta\in\Kmet(X,L)$ are related by~$\eta'=\eta+\dc\varphi+\vartheta$ for some~$\TorusL$-invariant and~$\xi_0$-basic~$\varphi\in\m{C}^\infty(N,\R)$ and~$\vartheta\in \ker\left(\dd\colon\Alt^1(N)\to\Alt^2(N)\right)$~\cite{BoyerGalicki}.
\end{remark}

\begin{remark}\label{rmk:Kahlsection}
    If~$\vartheta$ is closed and~$\xi_0$-basic, then for every~$\eta\in\Kmet(X,L)$ we have~$\ScalTW(\eta+\vartheta)=\ScalTW(\eta)=\pi^*\Scal(\Kahmap(\eta))$, so the total TW-scalar curvature and the total volume are constant on each of the sets
    \begin{equation}
        \set*{\alpha\in\Cmet(X,L)^{\TorusL} \suchthat \alpha(\xi_0)=f \text{ and } \Pi(\alpha)=\omega }.
    \end{equation}
    In other words, if we are interested in critical points of the Einstein-Hilbert functional it is sufficient to fix a section~$\sigma$ of~$\Kahmap\colon\Kmet(X,L)\to\chern_1(L)_+$ and consider only CR contact forms that lie in~$\Cmet_\sigma(X,L)\coloneqq\Sasmap_{\xi_0}^{-1}\left(\sigma(\chern_1(L)_+)\right)$.
\end{remark}

A fundamental result of Webster~\cite{Webster_connection} shows that the TW-scalar curvature of any~$\eta \in \Kmet(X,L)^{\TorusL}$ coincides with (the pullback of) the scalar curvature of~$\Kahmap(\eta)= \omega$,   
\begin{equation}
   \ScalTW(\eta) = \pi^*\Scal(\omega).
\end{equation}
This allows us to translate the cscTW equation on~$(N,\uI,\xi_0)$ as a weighted cscK equation 
\begin{equation}
    \ScalTW(u^{2/n}\eta)u^{q-1}=2q\Delta_B u+u\,\Scal(\omega)
\end{equation}
when~$\Kahmap(\eta)= \omega$ following the works~\cite{ApostolovCalderbank_weighted,ApostolovCalderbankLegendre_weighted}. Moreover, both the TW-scalar curvature and the volume of a CR contact form only depend on the transversal K\"ahler structure, so that the~$\EH$-functional can be thought as a functional on~$c_1(L)_+^{\TorusL}\times \m{C}^{\infty}(X,\R_{>0})^{\TorusL}$, namely 
\begin{equation}
    \EH(u,\omega) \coloneqq \EH(u^{2/n}\eta)=  \frac{\int_X\left(u^2\,\Scal(\omega)+2q\abs{\dd u}^2_{\omega}\right)\omega^{[n]}}{\norm{u}^2_{L^q(\omega)}}
\end{equation}
for any~$\eta\in (\Kahmap)^{-1}(\omega)$. 

\begin{remark}
    \emph{The general Sasaki setting}. As in~\cite[\S 3]{LahdiliLegendreScarpa}, the spaces~$\Kmet(X,L)$ and~$\Cmet(X,L)$ of Sasaki and pseudo-hermitian structures can be defined more generally without requiring the existence of a regular Sasaki-Reeb vector field and a smooth K\"ahler quotient. Indeed, given a transversal holomorphic structure~$(N,\uI,\xi)$, of Sasaki type, with the action of a compact torus~$\TorusL\subset\Aut(N,\uI)$ such that~$\xi\in \Lie(\TorusL)$, one can always define the space~$\Kmet(N,\uI,\xi)^{\TorusL}$ of~$\TorusL$-invariant Sasaki structure with Reeb vector field~$\xi$, and the space~$\Cmet(N,\uI,\xi)^{\TorusL}$ of all CR-contact forms conformal to an element of~$\Kmet^{\TorusL}$ as before. Since the Sasaki-Reeb cone of a Sasaki manifold necessarily contains a quasi-regular Reeb vector field then one can identify~$\Kmet(N,\uI,\xi)^{\TorusL}$ with~$\Kmet(X,L)^{\TorusL}$ where~$(X,L)$ is a \emph{polarized orbifold} in the sense of~\cite{RossThomas_orbifolds}.    
\end{remark}

\subsection{The Sasaki-Reeb cone and contact momentum map}\label{ss:ReebMoment}

\paragraph{The Reeb cone of a Sasakian torus action.}
Consider a regular Sasaki manifold $(N,D,\uI,\eta_0,{\xi_0})$ with K\"ahler quotient~$(X, L,\omega_0)$ such that~$\Kahmap(\eta_0)= \omega_0$, and let~$\TorusL\curvearrowright N$ be as above. In this context, a concise way to introduce the \emph{Sasaki-Reeb cone}, relative to~$\TorusL$, is
\begin{equation}
    \Rcone \coloneqq \set*{ \chi \in \torusL \suchthat \eta_0(\chi)>0 \text{ on } N }.
\end{equation}
Note that any~$\chi \in \Rcone$ is the Reeb vector field of~$\eta(\chi)^{-1}\eta$ for any~$\eta\in \Kmet(X,L)^{\TorusL}$ (c.f.\ Lemma~\ref{lemma:Killing}), as~$\eta(\chi)$ is a transversal Killing potential for~$\chi$ with respect to~$\eta$. In fact, the space of pseudo-hermitian structures whose Reeb vector field lies in the Lie algebra of~$\TorusL$ is indexed by the Sasaki-Reeb cone, namely
\begin{equation}
    \set*{ \alpha\in \Cmet(X,L)^{\TorusL} \suchthat R^\alpha\in \kt } \simeq \Kmet(X,L)^{\TorusL}\times \Rcone
\end{equation}
and all these structures are Sasaki since~$\bT\subset {\textrm Aut}(D,J)$.

\begin{remark}\label{rmk:propostional_sasaki}
    Expanding on this observation, assume that~$\eta_0$ and~$\eta_1$ are CR contact forms for~$(N,D,\uI)$ (in particular,~$D=\ker\eta_0=\ker\eta_1$) and are both Sasaki. Then, for any~$\zeta\in\Rcone$,~$\eta_0(\zeta)^{-1}\eta_0 = \eta_1(\zeta)^{-1}\eta_1$ since both sides are CR contact forms for~$(N,D,\uI)$, and they have the same Reeb field~$\zeta$.
\end{remark}

\paragraph{The Reeb cone and moment maps.}
By definition (c.f.\ Section~\ref{sec:torus_CR}), the action of~$\TorusL$ on~$L$ is a linearization of the (effective) action~$\TorusX\curvearrowright X$. For every~$a\in\torusL$, we denote by~$\underline{a}$ the infinitesimal action of~$a$ on either~$L$ or~$X$. Let~$\eta\in\Kmet(X,L)^{\TorusL}$, and set~$\omega=\Kahmap(\eta)$. A key fact is that for any~$a\in\torusX$, the function
\begin{equation}\label{eqMommap=eta}
    \langle \mu_\omega, a\rangle \coloneqq \eta(\underline{a})
\end{equation}
is a Hamiltonian for~$\underline{a}\in\Gamma(TX)$ with respect to~$\omega$. Moreover, the image of~$\mu_\omega \colon X \to \torusX^*$ does not depend on the choice of~$\eta\in (\Kahmap)^{-1}(\omega)$, see e.g.\ \cite{ApostolovCalderbankLegendre_weighted}, and is a compact convex polytope~$\check{P}_L\subset\torusX^*$~\cite{BGcontactToric,Lerman}, a moment polyotpe for the Hamiltonian action~$\TorusX\curvearrowright(X,\omega)$. Of course~\eqref{eqMommap=eta} extends to a map~$\mu_\eta:N\to\torusL^*$ by~$\langle \mu_\eta, a\rangle \coloneqq \eta(\underline{a})$, whose image is a compact convex polytope~$P_L\subset\torusL^*$. We will often consider~$P_L$ as the \emph{moment polytope} of a torus action rather than the more common (symplectic) moment polytope~$\check{P}_L$. Fixing a splitting~$\torusL\simeq\torusX\oplus\R\xi_0$ gives a way of identifying~$\check{P}_L$ with a section of~$P_L$. In any case, independently of such a splitting,~\eqref{eqMommap=eta} shows that \emph{the Sasaki-Reeb cone is naturally identified with the positive affine-linear functions on~$P_L$}.

\paragraph{The Einstein-Hilbert functional on Sasaki structures.}
By~\cite[Proposition~$4.4$]{FutakiOnoWang}, the value of the total TW-scalar curvature~$\Scaltot$ on the set of Sasaki structures only depends on the transverse K\"ahler structure, i.e.\ on the Sasaki-Reeb field. More precisely, for any~$\TorusL$-invariant Sasaki structures~$(N,D,\uI,\eta,\chi)$,~$(N,D',\uI',\eta',\chi)$ with the same~$\chi$-transversal holomorphic structure we have~$\Scaltot(\eta)=\Scaltot(\eta')$. In particular, for~$\chi\in\Rcone$ and any~$\eta\in\Kmet(X,L)^{\TorusL}$, we can define
\begin{equation}
    \Scaltot(\chi)\coloneqq\Scaltot(\eta(\chi)^{-1}\eta).
\end{equation}
The same holds for the total volume functional. Therefore, the Einstein-Hilbert functional can be seen as a functional on~$\Rcone$ that we denote now~$\REH$ to avoid confusion. That is, fixing~$(N,D,\uI,\eta,\xi_0)$ as above, we define
\begin{equation}
\begin{split}
    \REH \colon \Rcone  & \to \R \\
    \chi &\mapsto  \EH(\eta(\chi)^{-1}\eta).
\end{split}
\end{equation}
It is known from~\cite{EinsteinHilbert-SasakiFutaki} that the critical points of~$\REH$ coincide with the Sasaki-Reeb vector fields whose transversal Futaki invariant vanishes. Moreover, there is always at least one ray of Sasaki-Reeb vector fields minimizing~$\REH$ as a result of~\cite{BoyerHuangLegendre_DH}, so we define \begin{equation}\label{eq:EHmin}
    \REHmin \coloneqq \inf_{\chi\in\Rcone} \REH(\chi).
\end{equation}

\subsection{Detecting cscS structures with the Yamabe energy} \label{sec:Examples}

For every~$\eta\in\Kmet(X,L)^{\TorusL}$, we can consider the CR Yamabe energy~$\CRYamT$ of Proposition~\ref{prop:CRYam_solution}, and define
\begin{equation}
    \CRYamInvT(X,L) \coloneqq \sup_{\eta\in\Kmet(X,L)^{\TorusL}}  \CRYamT(\eta).
\end{equation}
By definition,~$\CRYamT(\eta) \leq \REH(\chi)$ for any~$\chi \in \Rcone$, so that 
\begin{equation}\label{eqEHmin=Ubound}
\CRYamInvT(X,L) \leq \REHmin.
\end{equation}

With all this in place, Theorem~\ref{thm:CRYam_equal_cscTW} follows easily from Theorem~\ref{thm:JerisonLee_CRYam}. We restate Theorem~\ref{thm:CRYam_equal_cscTW} more precisely as follows.
\begin{proposition}[Theorem~\ref{thm:CRYam_equal_cscTW}] \label{prop:CRYam_consequences}
    Let~$\eta\in\Kmet(X,L)^{\TorusL}$. If~$\eta$ satisfies~$\CRYam(\eta) =\REHmin$ then~$\eta(\chi)^{-1}\eta$ is cscS for any~$\REH$-minimizer~$\chi \in \Rcone$. Moreover, assuming~$\REHmin\leq 0$ the converse also holds: if~$\alpha\in[\eta]^{\TorusL}$ is cscS with Sasaki-Reeb field~$\chi\in\Rcone$, then~$\chi$ minimises~$\REH$ and~$\CRYamT(\eta)=\REHmin$.
\end{proposition}
\begin{proof}
    Let~$\chi\in\Rcone$ be a minimiser of~$\REH$. Recall that, by definition
    \begin{equation}\label{eq:CRYam_less_EH0}
        \CRYamT(\eta) = \inf_{f\in\m{C}(N,\R_>0)^{\TorusL}}\EH(f^{-1}\eta) \leq \EH(\eta(\chi)^{-1}\eta)=\REHmin.
    \end{equation}
    If some~$\eta$ realises the equality in~\eqref{eq:CRYam_less_EH0}, then~$\eta(\chi)$ realises~$\inf_{f\in\m{C}(N,\R_>0)^{\TorusL}}\EH(f^{-1}\eta)$, so
    that~$\eta(\chi)^{-1}\eta$ is a critical point of~$\EH\colon [\eta]^{\TorusL} \to \R$. In particular,~$\eta(\chi)^{-1}\eta$ is a cscTW form, and it is Sasaki with Reeb vector field~$\chi$, proving the first part.
    
    Assume now that~$\REHmin\leq 0$ and that~$\alpha\in[\eta]^{\TorusL}$ is cscS, with Reeb field~$\chi$. As~$\CRYamT(\eta)\leq\REHmin$, Theorem~\ref{thm:JerisonLee_CRYam} shows that~$\alpha=\inf_{[\eta]^{\TorusL}}\EH$ is the unique cscTW form in~$[\eta]^{\TorusL}$. Then we have the chain of inequalities
    \begin{equation}
        \CRYamT(\eta) \leq \REHmin \leq \EH(\chi)=\EH(\alpha)=\CRYamT(\eta).\qedhere
    \end{equation}
\end{proof}
This last result also implies that the Yamabe energy does not detect the cscS structures that are not minima of the Einstein-Hilbert functional.

\paragraph{Examples of cscS  undetectable by the Yamabe energy.} Example~$5.7$ of~\cite{EinsteinHilbert-SasakiFutaki} shows that the~$\REH$-functional on the Sasaki-Reeb cone of some join Sasaki manifolds over~$\mathbb{S}^3\times (\mathbb{P}^2\# k\mathbb{P}^2)$ for~$4\leq k\leq 8$ have two minimums and one local maximum.  This local maximum is therefore not a maximum of the Yamabe energy by~\eqref{eqEHmin=Ubound}. Note that these last examples were found via a root counting argument and are then only implicitly known. More explicit examples of multiple toric cscS on the circle bundle over~$\cp^1\times \cp^1$ are given in~\cite{Legendre_toricSasaki} but the values of the~$\REH$-functional on these example are not computed. We remedy to this now and show that two irregular cscS are not detectable by the Yamabe energy. 
Given coprime numbers~$p,q\in\N^*$, consider the rectangle~$P=P_{p,q}$ with vertices $$p_1\coloneqq(-p,-q), p_2\coloneqq(-p,q), p_3\coloneqq(p,q), p_4\coloneqq(p,-q)$$ as a moment polytope for the K\"ahler class~$p c_1(\mO(2))+ qc_1(\mO(2))$ on the product~$\cp^1\times \cp^1$. The Sasaki-Reeb cone~$\Rcone$ is identified with the space of the (strictly) positive affine-linear functions over~$P$ and below, to a subset of~$\R^3$, namely~$\Rcone= \{ (a,b,c) \in \R^3 \,|\, a+bx+cy >0,\,  \forall (x,y)\in P \}$. We recall the following.

\begin{proposition}\cite{Legendre_toricSasaki}\label{propExixt3cscS}
    For any pair of integers~$p,q\in \N^*$ with~$q>5p$, the circle bundle over~$\cp^1\times\cp^1$ corresponding to the K\"ahler class~$p c_1(\mO(2))+ qc_1(\mO(2))$ admits $3$ toric cscS structures~$(\Ds_0,J_0,\eta_0,\xi_0)$, $(\Ds_-,J_-,\eta_-,\xi_-)$,  and~$(\Ds_+,J_+,\eta_+,\xi_+)$. The corresponding Sasaki-Reeb vector fields are, up to a dilation, $\xi_0 =(1,0,0)$ and $$\xi_{\pm} = \left(1, \pm \frac{1}{p} \sqrt{1+ \frac{4q}{5(p-q)} },0\right).$$     
\end{proposition}

\begin{lemma}
    We have~$\REH(\xi_{\pm})>\REH(\xi_0)$.
\end{lemma}
\begin{proof}
    For~$\xi\in \Rcone$, we denote~$\xi_i=\xi(p_i)>0$. When~$\xi(0)=1$ (which we assume from now on), the volume of the Sasaki manifold associated to~$\xi$ is~$\bfV(\xi) = \frac{4pq}{\xi_1\xi_2\xi_3\xi_4}$ and the total Tanaka-Webster curvature is 
    \begin{equation}
        \bfS(\xi) = \frac{2q}{\xi_1\xi_2}+\frac{2p}{\xi_2\xi_3}+\frac{2q}{\xi_3\xi_4}+\frac{2p}{\xi_4\xi_1} = 2q\left(\frac{1}{\xi_1\xi_2}+\frac{1}{\xi_3\xi_4}\right)+2p\left(\frac{1}{\xi_2\xi_3}+\frac{1}{\xi_4\xi_1}\right).
    \end{equation}
    Writing~$\xi(x,y)= 1 +bx+cy$ and~$h(\xi)=h(b,c)\coloneqq \frac{1}{p}+ \frac{1}{q}+ b^2\left(p-\frac{p^2}{q}\right)+ c^2\left(q-\frac{q^2}{p}\right)$, the~$\REH$-functional writes $$\REH(\xi) \coloneqq\frac{\bfS(\xi)}{\bfV(\xi)^{2/3}} = h(b,c) \bfV(\xi)^{1-\frac{2}{3}}.$$

    When~$q>5p$,~$c=0$ and~$b= \frac{1}{p} \sqrt{1+ \frac{4q}{5(p-q)}}$ then for~$\xi_{\pm}=(1,\pm b,0)$ we have
    \begin{equation}
        \bfV(\xi_{\pm}) = \frac{4pq}{(1-p^2b^2)^2} =\frac{4pq}{16q^2}(25(p-q)^2)    
    \end{equation}
    and, since~$-1<\frac{4q}{5(p-q)}<0$, we find~$\bfV(\xi_0) = 4pq < \bfV(\xi_-) = \bfV(\xi_+) = 4pq \frac{25(\frac{p}{q}-1)^2}{16}$. Moreover, 
    \begin{equation}
    \begin{split}
        \frac{\REH(\xi_{+})}{\REH(\xi_0)} &= \frac{h(\xi_+) \bfV(\xi_+)^{1/3} \bfV(\xi_0)^{2/3}  }{\bfS(\xi_0)} = \frac{h(\xi_+) 4pq \left(\frac{25(p-q)^2}{16q^2}\right)^{1/3}}{4(p+q)} \\
        &= \left(1 + p^2b^2\frac{(q-p)}{p+q}\right)\left(\frac{25(\frac{p}{q}-1)^2}{16}\right)^{1/3} >1
    \end{split}
    \end{equation}
    because~$\frac{25(\frac{p}{q}-1)^2}{16} >1$ and~$p^2b^2\frac{(q-p)}{p+q}>0$.
\end{proof}

\section{Regularity results}\label{sec:regularity}

We start this Section by showing a differentiability result for the CR Yamabe energy, Theorem~\ref{thm:differentiabilityVersionVague}. We then proceed to examine, under some simplifying assumptions, the behaviour of the Einstein-Hilbert functional around a critical point~$\alpha\in\Cmet(X,L)$. We will show that the (non-)degeneracy of~$\Hess\EHT$ at~$\alpha$ is governed by the eigenvalues of the~$\alpha$-basic Laplacian~$\Delta_{B,\alpha}$ (which we denote also by~$\Delta_B$ or~$\Delta_\alpha$). By~\cite{JerisonLee_CRYamabe}, the basic Laplacian (also called \emph{sublaplacian}) of any~$\alpha\in\Cmet(X,L)$ is an elliptic operator on~$\m{C}^{\infty}(X,L)$, self-adjoint with respect to the~$L^2$-pairing induced by the measure~$\vol_\alpha=\alpha\wedge(\dd\alpha)^{[n]}$. In particular, it has positive eigenvalues.

\subsection{The derivative of the CR Yamabe energy}

To compute the first variation of~$\CRYamT$ we follow the same strategy as in the Riemannian case, see the proof of~\cite[Proposition~$2.2$]{Anderson_Yamabe_differentiable}. The two main differences are that we parametrise the set~$\Kmet(X,L)$ as in Remark~\ref{rmk:Kahlsection}, and that our equation is \emph{sub}critical, so the compactness results necessary for the proof are almost immediate: see the proof of Proposition~\ref{prop:CRYam_solution}. Most of the computations, however, are completely analogous, so we only sketch the more technical part of the argument for which we refer to~\cite{Anderson_Yamabe_differentiable}.

Recall from Remark~\ref{rmk:Kahlsection} that any~$\eta\in\Kmet(X,L)^{\TorusL}$ can be expressed as~$\eta=\eta_0+\dc\varphi+\vartheta$ for a fixed~$\eta_0$, a basic function~$\varphi$ and a basic and closed form~$\vartheta$. As the Einstein-Hilbert functional and the CR Yamabe energy are constant in the~$\vartheta$ direction, we can fix a section of~$\Kmet(X,L)^{\TorusL}\to\chern_1(L)_+^{\TorusL}$ and consider~$\CRYamT$ as a function on~$\chern_1(L)_+^{\TorusL}$. From this point of view, we can restate our differentiability result as follows.
\begin{theorem}\label{thm:CRYam_diff_text}
    Fix~$\eta_0\in\Kmet(X,L)^{\TorusL}$, let~$\omega_0=\Kahmap(\eta_0)$ and let~$\sigma:\chern_1(L)_+^{\TorusL}\to\Kmet(X,L)^{\TorusL}$ be the unique section of~$\Kahmap$ such that~$\sigma(\omega_0)=\eta_0$ and~$\sigma(\omega_0+\ddc\varphi)=\eta_0+\dc\varphi$.

    Assume that there exists a unique, unit-volume minimiser~$f_0^{-1}\eta_0$ of~$\EHT$ in the conformal class of~$\eta_0\in\Kmet(X,L)^{\TorusL}$. Then~$\CRYamT\circ\sigma$ is differentiable at~$\omega_0$, and
    \begin{equation}\label{eq:diffential_CRsigma}
        D(\CRYam\circ\sigma)_{\omega_0}(\dot{\varphi})=\int_X f_0^{-n-1}\left\langle\nablaminus f_0,\nablaminus\dot{\varphi}\right\rangle\omega_0^{[n]}
    \end{equation}
    where all metric quantities in~\eqref{eq:diffential_CRsigma} are computed with respect to~$\omega_0$.
\end{theorem}
\begin{proof}
    Let~$\{\eta_t\mid t>0\}\subset\Kmet(X,L)^{\TorusL}$ be a path of contact forms that smoothly converges to~$\eta_0$ as~$t\to 0^+$. Our goal is to compute the limit
    \begin{equation}\label{eq:limit_derivative}
        \lim_{t\to 0^+}\frac{\CRYam(\eta_t)-\CRYam(\eta_0)}{t}.
    \end{equation}
    To prove that the limit~\eqref{eq:limit_derivative} exists, we use the Jerison-Lee parametrisation of the conformal classes of~$\eta_t$. Fix a sequence of ($\TorusL$-invariant) conformal factors~$u_t$ such that, for every~$t>0$,
    \begin{equation}
        \CRYamT(\eta_t)= \inf_{u\in\m{C}^\infty(N,\R_{>0})^{\TorusL}} \EH\big(u^{2/n}\eta_t\big)= \EH\big(u_t^{2/n}\eta_t\big).
    \end{equation}
    Let us normalise the conformal factors by imposing that~$\int u_t^q\vol_{\eta_t}=1$ for all~$t\geq 0$. Then,~$u_t$ converges to~$u_0$ as~$t\to 0^+$: as any minimising sequence is bounded in~$W^{1,2}$ and compact in~$L^q$ (c.f.\ proof of Proposition~\ref{prop:CRYam_solution}), the sequence~$\{u_t\}$ has an accumulation point~$\underline{u}$ that is a minimiser of~$u\mapsto\EH(u^{2/n}\eta_0)$, and~$\underline{u}=u_0$ since~$\eta_0$ is regular.
    
    We consider the rescaled family of contact forms~$\alpha_t\coloneqq \norm{u_0}_{L^q(\eta_t)}^{-1}u_0^{2/n}\eta_t$ and the functions~$v_t\coloneqq \norm{u_0}^{n/2}_{L^q(\eta_t)}u_0^{-1}u_t$. Then both~$\alpha_t$ and~$v_t^{2/n}\alpha_t$ have unitary volume, and each~$v_t$ solves the cscTW equation
    \begin{equation}\label{eq:cscTW_smoothfamily}
        v_t\ScalTW(\alpha_t)+2q\Delta_{\alpha_t} v_t=v_t^{q-1}\,\CRYamT(\alpha_t).
    \end{equation}
    Note that we can linearise the Tanaka-Webster scalar curvature of~$\alpha_t$, as the convergence~$\alpha_t\to\alpha_0$ is smooth. Since~$\alpha_0$ is cscTW and of unit volume, there exists a linear operator~$\m{L}_{\alpha_0}$ such that
    \begin{equation}
        \ScalTW(\alpha_t) = \CRYamT(\alpha_0)+t\,\m{L}_{\alpha_0}(\dot{\alpha}_0)+O(t^2).
    \end{equation}
    We use this to rewrite~\eqref{eq:cscTW_smoothfamily} as
    \begin{equation}
        v_t\,\CRYamT(\alpha_0)+t\,v_t\,\m{L}_{\alpha_0}(\dot{\alpha}_0)+2q\Delta_{\alpha_t} v_t = v_t^{q-1}\,\CRYamT(\alpha_t)+O(t^2)
    \end{equation}
    Integrating with respect to~$\alpha_t$ we get
    \begin{equation}
        \int_N v_t\,\m{L}_{\alpha_0}(\dot{\alpha}_0)\vol_{\alpha_t}+O(t)=\frac{1}{t}\int_N\left(v_t^{q-1}\CRYamT(\alpha_t)-v_t\,\CRYamT(\alpha_0)\right)\vol_{\alpha_t}
    \end{equation}
    from which we obtain an expression for the incremental ratio of~$\CRYam$:
    \begin{equation}
    \begin{split}
        &\frac{\CRYamT(\alpha_t)-\CRYamT(\alpha_0)}{t}\int v_t^{q-1}\vol_{\alpha_t}=\\
        &=\int \left(v_t\,\m{L}_{\alpha_0}(\dot{\alpha}_0)-\CRYamT(\alpha_0)\frac{v_t^{q-1}-v_t}{t}\right)\vol_{\alpha_t}+O(t).
    \end{split}
    \end{equation}
    We claim that as~$t\to 0$,~$\int(v_t^{q-1}-v_t)\vol_{\alpha_t}\in o(t)$. As~$v_t$ tends to~$1$, this will imply
    \begin{equation}
        \lim_{t\to 0}\frac{\CRYamT(\eta_t)-\CRYamT(\eta_0)}{t}=\lim_{t\to 0}\frac{\CRYamT(\alpha_t)-\CRYamT(\alpha_0)}{t}=\int \m{L}_{\alpha_0}(\dot{\alpha}_0)\vol_{\alpha_0}.
    \end{equation}
    To show the claim, we make the simplifying assumption that~$v_t$ can be expanded as~$v_t=1+t\dot{v}+O(t^2)$. Then,
    \begin{equation}
        v_t^{q-1}-v_t=(q-2)t\dot{v}+O(t^2)
    \end{equation}
    so it will be sufficient to show that~$\int\dot{v}\vol_{\alpha_0}=0$. But this is guaranteed by the unitary volume assumptions on~$\alpha_t$ and~$v_t^{2/n}\alpha_t$:
    \begin{equation}
        0=\partial_{t=0}\int v_t^q\vol_{\alpha_t}=\int q\dot{v}\vol_{\alpha_0}+\int\partial_{t=0}\vol_{\alpha_t}=\int q\dot{v}\vol_{\alpha_0}.
    \end{equation}
    It may not be true in general that the path~$v_t$ is differentiable in~$t$, e.g.\ if there are multiple minimisers of~$\EHT$ on~$[\eta_t]^{\TorusL}$ for~$t\not=0$. This issue can be circumvented following the same argument of~\cite[Proposition~$2.2$]{Anderson_Yamabe_differentiable}. The argument is completely analogous, so we will not give further details.
    
    What we have so far shows the following: under the hypotheses of Theorem~\ref{thm:CRYam_diff_text}, the differential of~$\CRYamT$ at~$\eta_0$ is
    \begin{equation}\label{eq:CRYam_differential_implicit}
        (D\CRYamT)_{\omega}(\dot{\omega})=\int\m{L}_{\alpha_0}(\dot{\alpha})\vol_{\alpha_0}
    \end{equation}
    where~$\alpha_0$ is the contact form that minimises~$\EHT$ in the class of~$\eta_0$, and~$\dot{\alpha}=\partial_{t=0}\alpha_t$ for any smooth path of unit-volume contact forms~$\alpha_t$ converging to~$\alpha_0$.
    
    Recall that~$\m{L}_{\alpha_0}(\dot{\alpha})$ is the differential of~$\ScalTW(\alpha_t)$. The unit-volume assumption then allows to write
    \begin{equation}
        \int\m{L}_{\alpha_0}(\dot{\alpha})\vol_{\alpha_0}=\partial_{t=0}\Scaltot(\alpha_t)-\ScalTW(\alpha_0)\partial_{t=0}\Voltot(\alpha_t)=\partial_{t=0}\Scaltot(\alpha_t).
    \end{equation}
    Moreover, this coincides with the derivative of~$\EH$ along the path~$\alpha_t$, which we can express using the computations in~\cite[\S$2.3$, \S$3.1$]{LahdiliLegendreScarpa}.
    
    To make use of the results in~\cite{LahdiliLegendreScarpa}, we express~$\alpha_t$ in the Reeb parametrisation of the conformal class of a CR contact form,~$\alpha_t=f_t^{-1}(\eta_0+\dc\varphi_t)$ for some paths~$\{\varphi_t\}$,~$\{f_t\}$ of~$\TorusL$-invariant smooth functions on~$X$ such that~$\varphi_0=0$ and~$f_t>0$. Decomposing the variation of~$\EH$ in the conformal and ``K\"ahler'' directions we find, by~\cite[eq.~$(2.11)$ and Lemma~$3.5$]{LahdiliLegendreScarpa},
    \begin{equation}
    \begin{split}
        \partial_{t=0}\EH(\eta_t)=&\partial_{t=0}\EH(f_t^{-1}\eta_0)+\partial_{t=0}\EH(f_0^{-1}(\eta_t+\dc\varphi_t))=\\
        =& -n\int_N f_0^{-n-2}\dot{f}\left(\ScalTW(\alpha_0)-\Scaltot(\alpha_0)\right)\vol_{\eta_0}\\
        +n\int_N f_0^{-2} &\left[ \left\langle\nabla^-\dd f_0,\nabla^-\dd\dot{\varphi}\right\rangle_{\alpha_0} +\frac{1}{2}\left(\ScalTW(\alpha_0)-\Scaltot(\alpha_0)\right)\left\langle\dd f_0,\dd\dot{\varphi}\right\rangle_{\alpha_0}\right]\vol_{\alpha_0}\\
        =&n\int_N f_0^{-2} \left\langle\nabla^-\dd f_0,\nabla^-\dd\dot{\varphi}\right\rangle_{\alpha_0}\vol_{\alpha_0} =n\int_X f_0^{-n-1}\left\langle\nabla^-\dd f_0,\nabla^-\dd\dot{\varphi}\right\rangle_{\omega_0}\omega_0^{[n]}
    \end{split}
    \end{equation}
    where in the second-to-last last equality we used that~$\alpha_0$ is cscTW.
\end{proof}
\begin{remark}
    There might be other situations in which~$\CRYam$ is differentiable. For example, if~$\eta_t$ is a smooth path in~$\Kmet(X,L)$ and~$f_t$ is a smooth path of positive functions such that~$\CRYam(\eta_t)=\EH(f_t^{-1}\eta_t)$ for all~$t$, the proof of Theorem~\ref{thm:CRYam_diff_text} shows that~$\CRYam$ has a directional derivative along the path~$\eta_t$. Moreover, if~$\CRYam$ is differentiable at~$\eta$, the differential must equal~\eqref{eq:diffential_CRsigma}.
\end{remark}
To conclude the proof of Theorem~\ref{thm:differentiabilityVersionVague}, assume that~$\CRYamT$ is differentiable at~$\eta_0$ and that its differential vanishes; then~$\nabla^-\dd f_0=0$, so~$f_0$ is a Killing potential and~$f_0^{-1}\eta_0$ is Sasaki of constant Tanaka-Webster scalar curvature.

We conclude this subsection by showing that, under natural assumptions, every critical point of~$\CRYamT$ is a global maximum.
\begin{corollary}\label{lem:critCRYam_max}
    Assume that~$\TorusL$ is maximal, and that~$\CRYamT$ is differentiable at its critical point~$\eta \in \Kmet(X,L)^{\TorusL}$. Then~$\eta$ is a maximiser of~$\CRYamT$.
\end{corollary}
\begin{proof}
    By Theorem~\ref{thm:CRYam_diff_text}, there exists~$f$ such that~$f^{-1}\eta$ is cscS and~$\TorusL$-invariant. Denote by~$\chi$ the Reeb vector field of~$f^{-1}\eta$: by maximality,~$\chi\in\torusL$. This gives the claim since for all~$\tilde{\eta}\in\Kmet(X,L)^{\TorusL}$,
    \begin{equation}
        \CRYamT(\eta) = \EH(f^{-1}\eta)= \REH(\chi) \geq \CRYamT(\tilde{\eta}).\qedhere
    \end{equation}
\end{proof}

\subsection{The Hessian of EH at a critical point}\label{sec:Hessian}

Recall the general expression for the first variation of~$\EH$ along a path of contact forms~$\eta_\varepsilon\coloneqq(f+\varepsilon\dot{f})^{-1}(\eta+\varepsilon\dc_B\dot{\varphi})+O(\varepsilon^2)$ (with~$\eta$ a Sasaki form) from~\cite[\S$2$ and \S$3$]{LahdiliLegendreScarpa}:
\begin{equation}\label{eq:firstvariation_EH}
\begin{split}
    \frac{1}{n}\Voltot(\eta_0)^{\frac{n}{n+1}}\partial_{\varepsilon=0}\EH(\eta_\varepsilon)= & 2\int_N f^{-2}\langle(\nablaminus)_{B,\eta} f,(\nablaminus)_{B,\eta}\dot{\varphi}\rangle_{\eta_0}\vol_{\eta_0} \\
    - & \int_N f^{-2}\langle\dd f,\dd\dot{\varphi}\rangle_{\eta_0}\left(\ScalTW(\eta_0)-\frac{\Scaltot(\eta_0)}{\Voltot(\eta_0)}\right)\vol_{\eta_0}\\
    - & \int_N \frac{\dot{f}}{f}\left(\ScalTW(\eta_0)-\frac{\Scaltot(\eta_0)}{\Voltot(\eta_0)}\right)\vol_{\eta_0}.
\end{split}
\end{equation}
It is easy to see from this expression that~$\eta_0=f^{-1}\eta$ is a critical point of~$\EH$ if and only if it is a cscS structure, see~\cite[Theorem~$1.1$]{LahdiliLegendreScarpa}. In particular, the Sasaki-Reeb field~$R^{\eta_0}$ of~$\eta_0$ generates a holomorphic torus action. For the rest of this Section, we assume that~$\eta_0$ is cscS,~$\Voltot(\eta_0)=1$, and that~$\TorusL\curvearrowright N$ is a torus action as in \S\ref{sec:torus_CR} such that~$\Lie(\TorusL)$ contains both~$\xi_0$ and~$R^{\eta_0}$. To compute the Hessian of~$\EHT$ at~$\eta_0$, we will need an expression for~$\partial_{t=0}\left(\nabla_t^-\!\dd_B f_t\right)$, for a path of CR contact forms~$f_t^{-1}\eta_t$. We can follow~\cite[Chapter~$1$]{PGbook}.
\begin{lemma}\label{lemma:Lich_variation}
    Assume that~$\eta$ is Sasaki. For any path~$\alpha_t=f_t^{-1}\eta_t$ such that~$\eta_t$ are Sasaki,~$\eta_0=\eta$,~$f_0=1$,
    \begin{equation}
        \partial_{t=0}\left((\nablaminus)_{B,\alpha_t}f_t\right)=(\nablaminus)_{B,\eta}\dot{f}_0.
    \end{equation} 
\end{lemma}
\begin{proof}
    As this is a local computation, we can work with the transversally K\"ahler structure of~$\eta_t$,~$(\dd\eta_t,J)$ on~$\ker\eta_t$. In particular we can use the characterization in \cite[Lemma~$1.23.2$]{PGbook}
    \begin{equation}
        (\nablaminus)_{B,\alpha_t} f_t = -\frac{1}{2}g_t(J\mL_{\nabla_t f_t}J\cdot, \cdot).
    \end{equation}
    As~$f_0=1$, the same characterisation gives
    \begin{equation}
        \partial_{t=0}\left((\nablaminus)_{B,\alpha_t} f_t\right) = -\frac{1}{2}g(J\mL_{\nabla\dot{f}_0}J\cdot, \cdot) = (\nablaminus)_{B,\eta}\dot{f}_0.\qedhere
    \end{equation}
\end{proof}

By Remark~\ref{rmk:Kahlsection0} and~\ref{rmk:Kahlsection}, to compute the derivatives of the Einstein-Hilbert functional, it is sufficient to consider its variation along paths in~$\Cmet(X,L)$ of the form
\begin{equation}\label{eq:contactforms_path}
    \alpha_t=\left(f_{t}^{-1}(\alpha+\dc_B\varphi_{t})\right)=\alpha+t\left(-\dot{f}\alpha+\dc_B\dot{\varphi}\right)+O(t^2).    
\end{equation}
For ease of notation, we will identify the vectors~$\dot{\alpha}\in T_\alpha\Cmet(X,L)$ given by a path as in~\eqref{eq:contactforms_path} with the pair~$(\dot{f},\dot{\varphi})$.

\begin{proposition}\label{prop:HessianEH_full}
    Let~$\eta_0\in\Cmet(X,L)$ be a Sasaki form of curvature~$\ScalTW(\eta_0)=c_0\in\R$ and of unit volume, and assume that~$R^{\eta_0}\in\Lie(\TorusL)$. Consider two tangent directions~$v_1,v_2\in T_{\eta_0}\set*{\alpha\in\Cmet(X,L)^{\TorusL}\suchthat\Voltot(\alpha)=1}$,~$v_j\coloneqq(\dot{f}_j,\dot{\varphi}_j)$ for~$j=1,2$. Then,
    \begin{equation}
    \begin{split}
        (\Hess\EHT)_\eta(v_1,v_2)=& 2\int_N\left(\langle\nablaminus\dot{f}_2,\nablaminus\dot{\varphi}_1\rangle+\langle\nablaminus\dot{f}_1,\nablaminus\dot{\varphi}_2\rangle\right)\vol_{\eta_0}\\
        & + 2(n+1)\int_N\langle\dd\dot{f}_1,\dd\dot{f}_2\rangle\vol_{\eta_0}-c_\eta\int_N\dot{f}_1\dot{f}_2\vol_{\eta_0}.
    \end{split}
    \end{equation}
\end{proposition}
\begin{proof}
    Consider any path of conformal factors~$f_{x,y}\coloneqq 1+x\dot{f}_1+y\dot{f}_2$ and any path of Sasaki forms~$\eta_{x,y}\coloneqq \eta_0+x\dc_B\dot{\varphi_1}+y\dc_B\dot{\varphi_2}$ (up to higher order terms in~$x$ and~$y$), defining the path~$\alpha(x,y)\coloneqq f_{x,y}^{-1}\eta_{x,y}$, which we assume to be of unitary volume. Then, as~$\alpha(0,0)=\eta_0$ is cscS, we compute the second variation of~$\EH$ as follows (note that~$f_{(0,0)}=1$)
    \begin{equation}
    \begin{split}
        \partial_y\partial_x\EHT(\alpha(x,y))\Bigr|_{(x,y)=(0,0)} =& 2\int_N\left\langle\partial_{y=0}\left(\nabla_{(0,y)}^-\!\dd f_{(0,y)}\right),\nablaminus\dot{\varphi}_1\right\rangle\vol_{\eta_0}\\
        &-\int_N\dot{f}_1\partial_{y=0}\left(\ScalTW(\eta_{0,y})-\Scaltot(\eta_{0,y})\right)\vol_{\eta_0}.
    \end{split}
    \end{equation}
    By Lemma~\ref{lemma:Lich_variation} we can rewrite the first term as
    \begin{equation}
        \int_N\left\langle\partial_{y=0}\left(\nabla_{(0,y)}^-\!\dd f_{(0,y)}\right),\nablaminus\dot{\varphi}_1\right\rangle\vol_{\eta_0} = \int_N\left\langle\nablaminus\dot{f}_2,\nablaminus\dot{\varphi}_1\right\rangle\vol_{\eta_0}.
    \end{equation}
    For the second term, we instead start from~\eqref{eq:Tanno} and~\eqref{eq:Scaltot_Reeb} to get
    \begin{equation}
    \begin{split}
        \int_N&\dot{f}_1\partial_{y=0}\left(\ScalTW(\eta_{0,y})-\Scaltot(\eta_{0,y})\right)\vol_{\eta_0} = \\
        &=\int_N\dot{f}_1\left(\dot{f}_2 c_0-2\mathbb{L}_{\eta_0}(\dot{\varphi}_2)-2(n+1)\Delta_{\eta_0}\dot{f}_2\right)\vol_{\eta_0}-c_0\int_N\dot{f}_1\vol_\eta\int_N\dot{f}_2\vol_{\eta_0}
    \end{split}    
    \end{equation}
    but as we are assuming that~$v_1,v_2$ are tangent to the space of volume-normalised contact forms,~$\int\dot{f}_j\vol_{\eta_0}=0$ for~$j=1,2$. Putting together these terms gives the thesis.
\end{proof}

\begin{corollary}[Theorem~\ref{thm:isolation}]\label{cor:isolation}
    Under the hypotheses of Proposition~\ref{prop:HessianEH_full}, assume that~$\TorusL$ is maximal and that the first non-zero eigenvalue of the Laplacian~$\Delta_{\eta_0}$ on~$\TorusL$-invariant functions,~$\lambda_1^{\TorusL}(\eta_0)$, satisfies
    \begin{equation}\label{eq:eigenval_condition}
        \lambda_1^{\TorusL}(\eta_0)>\frac{c_0}{2(n+1)}.   
    \end{equation}
    Then every cscS structure in a neighbourhood of~$\eta_0$ in~$\Cmet(X,L)^{\TorusL}$ is in the~$\TorusL^{\C}$-orbit of~$\eta_0$.
\end{corollary}
\begin{proof}
    With the notation of Proposition~\ref{prop:HessianEH_full}, assume that~$v_1$ is tangent to the space of unit-volume~$\TorusL$-invariant forms and lies in the kernel of~$(\Hess\EHT)_{\eta_0}$. Then,
    \begin{equation}\label{eq:kernel_cond}
    \begin{split}
        &2\int_N\left(\langle\nablaminus\dot{f}_2,\nablaminus\dot{\varphi}_1\rangle+\langle\nablaminus\dot{f}_1,\nablaminus\dot{\varphi}_2\rangle\right)\vol_{\eta_0}\\
        & + 2(n+1)\int_N\langle\dd\dot{f}_1,\dd\dot{f}_2\rangle\vol_{\eta_0}-c_0\int_N\dot{f}_1\dot{f}_2\vol_{\eta_0}=0
    \end{split}
    \end{equation}
    for every~$\TorusL$-invariant functions~$\dot{f}_2,\dot{\varphi}_2$ of zero average.
    
    In particular, if we choose~$\dot{f}_2=0$, we deduce that~$\nablaminus\dot{f}_1=0$. Choosing instead~$\dot{f}_2=\dot{f}_1$ we obtain
    \begin{equation}\label{eq:kernel_ass}
        2(n+1)\norm{\dd\dot{f}_1}_{L^2(\eta_0)}-c_0\norm{\dot{f}_1}_{L^2(\eta_0)}=0.
    \end{equation}
    The Rayleigh min-max characterization of~$\lambda_1^{\TorusL}(\eta_0)$ gives
    \begin{equation}
        2(n+1)\int_N\abs{\dd\dot{f}_1}^2_{\eta_0} \vol_{\eta_0} - c_0\int_N\dot{f}_1^2\vol_{\eta_0} \geq \left(2(n+1)\lambda_1^{\TorusL}({\eta_0})-c_0\right) \int_N\dot{f}_1^2\vol_{\eta_0}
    \end{equation}
    (with equality if and only if~$\Delta_{\eta_0}\dot{f}_1=\lambda_1^{\TorusL}(\eta_0)\dot{f}_1$). Then, putting together~\eqref{eq:kernel_ass} with condition~\eqref{eq:eigenval_condition} shows that~$\dot{f}_1=0$.
    
    Hence,~\eqref{eq:kernel_cond} shows that~$\dot{\varphi}_1$ lies in the kernel of~$\nablaminus$. As~$\TorusL$ is maximal and~$\dot{\varphi}_1$ is~$\TorusL$-invariant,~$J\nabla\dot{\varphi}_1\in\Lie(\TorusL)$, so~$v_1$ belongs to the tangent space to the~$\TorusL^{\C}$-orbit of~$\eta_0$.
\end{proof}

\paragraph{Hessian of the Yamabe energy.} We now use Proposition~\ref{prop:HessianEH_full} to compute the Hessian of the \emph{Yamabe energy} around a critical point under some simplifying assumptions. Assume that~$\eta_0\in\Cmet(X,L)$ is Sasaki of unit volume and constant transversal scalar curvature, and that~$\CRYam$ is differentiable at~$\eta_0$, so that~$(\dd\CRYam)_{\eta_0}=0$ by Theorem~\ref{thm:CRYam_diff_text}.

Assume first that we have a smooth path of Sasaki structures~$\eta_{t}=\eta_0+t\dc_B\dot{\varphi}+O(t^2)$, and that there is a corresponding~$\m{C}^1$ path~$f_{t}\in\m{C}^\infty(N,\R_{>0})$ of unit volume that minimise~$\EH$ in the conformal classes~$[\eta_{t}]$. In particular,~$f_{0}=1$. The cscTW equation implies that, for some~$c_{t}\in\R$,
\begin{equation}
    c_{t} = f_{t}\ScalTW(\eta_{t}) -2(n+1)\Delta_{\eta_{t}}f_{t}-(n+2)(n+1)f^{-1}_{t}\abs{\dd f_{t}}^2_{\eta_{t}}.
\end{equation}
Taking the variation along~$t$ and evaluating at~$t=0$ gives (using that~$f_0\equiv 1$)
\begin{equation}\label{eq:CSCTWdel}
    \dot{c} = c_0 \dot{f} - 2\mathbb{L}(\dot{\varphi})-2(n+1)\Delta\dot{f},
\end{equation}
where~$\mathbb{L}=(\nablaminus)^*(\nablaminus)$ is the (transversal) Lichnerowicz operator of~$\eta_0$.

To compute the Hessian of~$\CRYam$, we consider two such paths~$\eta_{1,t},\eta_{2,t}$ with corresponding~$f_{1,t},f_{2,t}$. Setting~$v_j=\partial_{t=0}(f_{j,t}^{-1}\eta_{j,t})$ we have
\begin{equation}
    (\Hess\CRYam)_\eta(v_1,v_2)=(\Hess\EH)_\eta(v_1,v_2)
\end{equation}
so that Proposition~\ref{prop:HessianEH_full} gives, using~\eqref{eq:CSCTWdel}, that up to multiplication by a positive constant,
\begin{equation}
\begin{split}
    (\Hess\CRYam)_{\eta_0}(v_1,v_2)=&\int_N\left(\dot{f}_2 \left(c_0\dot{f}_1-2(n+1)\Delta\dot{f}_1\right)+\dot{f}_1\left(c_0\dot{f}_2-2(n+1)\Delta\dot{f}_2\right)\right)\vol_{\eta_0}\\
    &+ 2(n+1)\int_N\langle\dd\dot{f}_1,\dd\dot{f}_2\rangle\vol_{\eta_0}-c_0\int_N\dot{f}_1\dot{f}_2\vol_{\eta_0} =\\
    =& -2(n+1)\int_N\langle\dd\dot{f}_1,\dd\dot{f}_2\rangle\vol_{\eta_0} + c_0 \int_N\dot{f}_1\dot{f}_2\vol_{\eta_0}.
\end{split}
\end{equation}
As in Corollary~\ref{cor:isolation}, we conclude that if~$\TorusL$ is maximal and~$c_0/2(n+1)$ is not an eigenvalue of~$\Delta_{\eta_0}$, the isotropic directions of~$(\Hess\CRYam)_{\eta_0}$ are given by the~$\TorusL^{\C}$-orbit of~$\eta_0$.

Note also that this computation gives a local version of Corollary~\ref{lem:critCRYam_max}. 

\subsection{The space of cscTW structures}

Fix a torus action~$\TorusL$ as in Section~\ref{sec:torus_CR}. We denote by~$\cscTWsetT(X,L)$ the set of all CR contact forms in~$\Cmet(X,L)^{\TorusL}$ that are cscTW. By Proposition~\ref{prop:CRYam_solution}, there are~$\TorusL$-invariant minimisers of the~$\EH$ functional in the~$\TorusL$-invariant conformal class of any~$\eta\in\Kmet(X,L)^{\TorusL}$. In fact, for any~$\eta\in\Kmet(X,L)^{\TorusL}$ there exists at least a ray~$\set*{c\,f^{-1}\eta \suchthat c\in\R_{>0}}$ contained in~$\cscTWsetT(X,L)$.

Our last regularity result shows that the set~$\cscTWsetT(X,L)$ is generically a manifold. This is the CR analogue of~\cite[Theorem~$2.5$]{Koiso_spaceofmetrics}, see also~\cite[Theorem~$4.44$]{Besse}.
\begin{proposition}\label{prop:manifold}
    The set~$\cscTWsetT(X,L)$ is an infinite-dimensional manifold in a neighbourhood of any~$\alpha\in\cscTWsetT(X,L)$ of volume~$1$ for which~$\ScalTW(\alpha)/2(n+1)$ is not an eigenvalue of the~$\alpha$-basic Laplacian.
\end{proposition}
To prove Proposition~\ref{prop:manifold}, the key property is that the critical points of~$\EHT_{\restriction[\alpha]^{\TorusL}}$ are \emph{nondegenerate}, under the hypotheses of Proposition~\ref{prop:manifold}.

We can use the same computation of Proposition~\ref{prop:HessianEH_full} to obtain an expression for the Hessian, at a critical point, of~$\EHT$ restricted to a conformal class.
\begin{lemma}\label{lemma:HessianEH_vertical}
    Let~$\alpha$ be a critical point of~$\EH_{\restriction[\alpha]^{\TorusL}}$, a cscTW form with~$\ScalTW(\alpha) = c_\alpha$. Then, the Hessian of~$\EH_{\restriction[\alpha]^{\TorusL}}$ at~$\alpha$ evaluated on~$u,v\in T_\alpha \Cmet(X,L)^{\TorusL}$ is
    \begin{equation}
        \Hess(\EH_{\restriction[\alpha]^{\TorusL}})(u,v) = \frac{n}{\Voltot(\alpha)^{\frac{n}{n+1}}}\left\langle \left(2(n+1)\Delta_\alpha-c_\alpha \right)u^0, v^0\right\rangle_{L^2(\vol_\alpha)},
    \end{equation}
    where~$u^0$,~$v^0$ are the zero-average normalizations of~$u$ and~$v$, respectively.
\end{lemma}
We can not simply use the result of Proposition~\ref{prop:HessianEH_full}, as in Lemma~\ref{lemma:HessianEH_vertical}~$\alpha$ is not assumed to be Sasaki. However, the computation is completely analogous, starting from~\eqref{eq:firstvariation_EH}.

In fact,~$\EHT$ is \emph{convex} around its critical points under the condition~\eqref{eq:eigenval_condition} on the first eigenvalue of the Laplacian, a phenomenon that also appears in the classical Yamabe setting~\cite{LimaPiccioneZedda}. This generalises~\cite[Theorem~$1.7$]{EinsteinHilbert-SasakiFutaki}.
\begin{corollary}\label{cor:VertHess_eigenvals}
    With the same hypotheses of Lemma~\ref{lemma:HessianEH_vertical}, if~$c_\alpha=\ScalTW(\alpha)$ is not an eigenvalue of~$\Delta_\alpha$, the Hessian of~$\EH_{\restriction[\alpha]^{\TorusL}}$ at~$\alpha$ is nondegenerate. Moreover,~$\alpha$ is a local strict minimum of~$\EH_{\restriction[\alpha]^{\TorusL}}$ if and only if the first non-zero eigenvalue~$\lambda_1^{\TorusL}(\alpha)$ of the Laplacian~$\Delta_\alpha$ on~$\TorusL$-invariant functions satisfies~\eqref{eq:eigenval_condition}. 
\end{corollary}
\begin{proof}
    The first statement is an immediate consequence of Lemma~\ref{lemma:HessianEH_vertical}. For the second part: for any~$\TorusL$-invariant function~$u$,
    \begin{equation}
        \Hess(\EH_{\restriction[\alpha]^{\TorusL}})_\alpha(u,u) =\frac{n}{\Voltot(\alpha)^{\frac{n}{n+1}}}\left(2(n+1)\norm{\dd u^0}^2_{L^2(\vol_\alpha)} - c_\alpha\norm{u^0}^2_{L^2(\vol_\alpha)}\right),
    \end{equation}
    and as in Corollary~\ref{cor:isolation}, the min-max characterization of~$\lambda_1^{\TorusL}(\alpha)$ gives
    \begin{equation}
        2(n+1)\int_N\abs{\dd u^0}^2_\alpha \vol_\alpha - c_\alpha\int_N(u^0)^2\vol_\alpha \geq \left(2(n+1)\lambda_1^{\TorusL}(\alpha)-c_\alpha\right) \int_N(u^0)^2\vol_\alpha
    \end{equation}
    with equality if and only if~$\Delta_\alpha u^0=\lambda_1^{\TorusL}(\alpha) u^0$. Since the sequence of Laplacian eigenvalues tends to~$+\infty$ we see that the only way for the second derivative of~$\EH_{\restriction[\alpha]^{\TorusL}}$ to be sign definite is to be positive. From the last computation this happens if and only if~$\lambda_1^{\TorusL}(\alpha) > c_\alpha/2(n+1)$. 
\end{proof}

\section{Manifolds of negative average curvature}\label{sec:negativecurv}

The goal of this Section is to prove Theorem~\ref{thm:CRYam_formula}. We start with a result of Dietrich on CR manifolds, that we slightly adapt in order to take into account the torus action. See also~\cite[Theorem~$3.1$]{SungTakeuchi_noneq_CR} for a version with different~$L^p$ norms. Throughout this Section, we consider a torus action~$\TorusL\curvearrowright N$ as in Section~\ref{sec:torus_CR}.

\begin{proposition}[\cite{Dietrich_CRYam,SungTakeuchi_noneq_CR}]\label{prop:negative_CRYam}
    Assume that~$\eta_0$ is a~$\TorusL$-invariant CR contact form on~$N$. If~$\CRYamT(\eta_0) \leq 0$ then
    \begin{equation}\label{eq:negcurv_Yamineq}
        \abs*{\CRYamT(\eta_0)}\leq\norm*{\ScalTW(\eta)}_{L^{n+1}(\eta)}
    \end{equation}
    for any other~$\TorusL$-invariant~$\eta$ such that~$\ker\eta_0=\ker\eta$. Moreover, equality holds if and only if~$\eta$ is cscTW.
\end{proposition}
\begin{proof}
    Given~$\TorusL$-invariant contact forms~$\eta_0,\eta$ with the same kernel, there exists a positive~$\TorusL$-invariant function~$u$ such that~$\eta=u^{2/n}\eta_0$, and the Tanaka-Webster scalar curvatures of the two forms are related by
    \begin{equation}
        \ScalTW(\eta)=u^{1-q}\left(u\ScalTW(\eta_0)+2q\Delta_B u\right)
    \end{equation}
    where~$\Delta_B$ is the basic Laplacian of~$\eta_0$ and as always~$q=2(1+1/n)$. As the left hand side of~\eqref{eq:negcurv_Yamineq} does not depend on the choice of~$\eta_0$ in its conformal class, we can assume that~$\ScalTW(\eta_0)$ is a constant, and~$u=1$ is then the unique solution, up to constant rescaling, of the cscTW equation
    \begin{equation}
        \ScalTW(\eta_0)+2q\Delta_B u = c u^{1-q}.
    \end{equation}
    Letting~$\vol_\eta\coloneqq \eta\wedge(\dd\eta)^{[n]}$, we have
    \begin{equation}
    \begin{split}
        \int\ScalTW(\eta)u^{q-2}\vol_{\eta_0} = & \int \left(\ScalTW(\eta_0)+2q\,u^{-1}\Delta_B u\right)\vol_{\eta_0}\\
        \leq & \int\ScalTW(\eta_0)\vol_{\eta_0}
    \end{split}
    \end{equation}
    with equality if and only if~$u$ is a constant. On the other hand, H\"older's inequality gives
    \begin{equation}
    \begin{split}
        \int\abs*{\ScalTW(\eta)u^{q-2}}\vol_{\eta_0} \leq & \norm*{\ScalTW(\eta_0) u^{q-2}}_{L^{n+1}(\eta_0)}\left(\int\vol_{\eta_0}\right)^{\frac{n}{n+1}}.
    \end{split}
    \end{equation}
    As we are assuming that~$\CRYamT(\eta_0)\leq 0$ and~$\eta_0$ minimises the Einstein-Hilbert functional,~$\ScalTW(\eta_0)$ is a non-positive constant. Hence,
    \begin{equation}
    \begin{split}
        \norm*{\ScalTW(\eta_0) u^{q-2}}_{L^{n+1}(\eta_0)}\left(\int\vol_{\eta_0}\right)^{\frac{n}{n+1}} \geq& -\int\ScalTW(\eta_0)\vol_{\eta_0}\\
        =& \abs*{\CRYamT(\eta_0)}\left(\int\vol_{\eta_0}\right)^{\frac{n}{n+1}}.
    \end{split}
    \end{equation}
    To get the thesis, note that~$(q-2)(n+1)=q$, so that
    \begin{equation}
        \norm*{\ScalTW(\eta_0) u^{q-2}}^{n+1}_{L^{n+1}(\eta_0)}
        =\int\abs*{\ScalTW(\eta)}^{n+1}\vol_{\eta}.\qedhere
    \end{equation}
\end{proof}

\begin{corollary}[Theorem~\ref{thm:CRYam_formula}]\label{cor:CRYam_formula}
    If~$\CRYamInv^{\TorusL}(X,L)\leq 0$ (e.g.\ if~$\REHmin\leq 0$),
    \begin{equation}\label{eq:negative_CRYam_contact}
        \abs{\CRYamInv^{\TorusL}(X,L)}=\inf_{\alpha\in\Cmet(X,L)^{\TorusL}} \norm*{\ScalTW(\alpha)}_{L^{n+1}(\alpha)}.
    \end{equation}
\end{corollary}
\begin{proof}
    For any~$\eta\in\Kmet(X,L)^{\TorusL}$, choosing the~$\TorusL$-invariant conformal factor that realises the equality in Proposition~\ref{prop:negative_CRYam} gives 
    \begin{equation}
        \abs*{\CRYamT(\eta)}=\inf_{u>0}\norm*{\ScalTW(u^{2/n}\eta)}_{L^{n+1}(u^{2/n}\eta)}.
    \end{equation}
    Taking the infimum over~$\Kmet(X,L)^{\TorusL}$ of both sides then we get
    \begin{equation}\label{eq:CRYam_negative_intermediate}
        \inf_{\alpha\in\Cmet(X,L)^{\TorusL}}\norm*{\ScalTW(\alpha)}_{L^{n+1}(\alpha)}=\inf_{\eta\in\Kmet(X,L)}\abs*{\CRYamT(\eta)}=\inf_{\eta\in\Kmet(X,L)}\left(-\CRYamT(\eta)\right)
    \end{equation}
    because every~$\CRYam(\omega)$ is negative. Hence,
    \begin{equation}
        \inf_{\alpha\in\Cmet(X,L)^{\TorusL}}\norm*{\ScalTW(\alpha)}_{L^{n+1}(\alpha)}=-\sup_{\eta\in\Kmet(X,L)^{\TorusL}}\CRYamT(\eta)=-\CRYamInv^{\TorusL}(X,L).\qedhere
    \end{equation}
\end{proof}

We now consider the existence of approximate cscS structures in~$\Cmet(X,L)^{\TorusL}$ in the negative average curvature case~$\REHmin\leq 0$.

\begin{corollary}\label{cor:approx_cscS}
    Assume that~$\REHmin\leq 0$. Then the following are equivalent:
    \begin{enumerate}
        \item for every~$\varepsilon>0$ there exists~$\alpha_\varepsilon\in\Cmet(X,L)^{\TorusL}$ of unit volume such that
        \begin{equation}
            \norm*{\ScalTW(\alpha_\varepsilon)-\REHmin}_{L^{n+1}(\alpha_\varepsilon)}<\varepsilon;
        \end{equation}
        \item~$\CRYamInv^{\TorusL}(X,L)=\REHmin$;
        \item for every~$\varepsilon>0$ there exists~$\alpha_\varepsilon\in\Cmet(X,L)^{\TorusL}$ that is cscTW, of unit volume, and satisfies
        \begin{equation}
            \abs*{\ScalTW(\alpha_\varepsilon)-\REHmin}<\varepsilon.
        \end{equation}
    \end{enumerate}
    In particular, if~$(X,L)$ admit~$L^p$-approximate cscS structures for some~$p\in[n+1,\infty]$ (c.f.\ Definition~\ref{def:approx_cscS}) then~$\CRYamInv^{\TorusL}(X,L)=\REHmin$.
\end{corollary}

\begin{proof}
    Assume~$(1)$. From Corollary~\ref{cor:CRYam_formula}, we see that for every~$\varepsilon$
    \begin{equation}
        -\CRYamInv^{\TorusL}(X,L) \leq \inf_{u>0} \norm*{2q\,u^{-1}\Delta_\varepsilon u+\ScalTW(\alpha_\varepsilon)}_{L^{n+1}(\alpha_\varepsilon)}\leq \norm*{\ScalTW(\alpha_\varepsilon)}_{L^{n+1}(\alpha_\varepsilon)}
    \end{equation}
    by choosing~$u$ to be any constant. As the~$\alpha_\varepsilon$ have unitary volume,
    \begin{equation}
        -\CRYamInv^{\TorusL}(X,L)\leq\norm*{\ScalTW(\alpha_\varepsilon)-\REHmin}_{L^{n+1}(\alpha_\varepsilon)}+\abs{\REHmin}.
    \end{equation}
    This shows that~$\CRYamInv^{\TorusL}(X,L)\geq \REHmin - \varepsilon$, and~$(2)$ follows since~$\CRYamInv^{\TorusL}(X,L)\leq \REHmin$. The other implications~$(2)\Rightarrow(3)$ and~$(3)\Rightarrow(1)$ follow directly from the definitions.
\end{proof}

\subsection{Existence of approximate cscS structures}\label{sec:weakconverse}

Expressing the CR Yamabe energy as in Proposition~\ref{prop:negative_CRYam} also allows us to make partial progress towards a converse of Corollary~\ref{cor:approx_cscS}: assuming that~$\CRYamInv^{\TorusL}(X,L)=\REHmin \leq 0$, we aim to show that there exist~$L^p$-approximate csc\emph{S} structures for some~$p\in[1,\infty]$.

The general strategy is as follows: let~$\chi_{\min}$ be a Reeb vector field minimising~$\REH$, and fix a section~$\sigma$ of~$\Kahmap\colon\Kmet(X,L)^{\TorusL}\to\chern_1(L)_+^{\TorusL}$. Assuming~$\CRYamInv^{\TorusL}(X,L)=\REHmin$, for every~$\varepsilon$ there is~$\eta_\varepsilon$ in the image of~$\sigma$ such that
\begin{equation}
    \REHmin-\varepsilon\leq \CRYamT(\eta_\varepsilon) \leq \REHmin.
\end{equation}
For every~$\varepsilon$,~$\alpha_\varepsilon\coloneqq\eta_\varepsilon(\chi_{\min})^{-1}\eta_\varepsilon$ is a Sasaki form with Reeb vector field~$\chi_{\min}$ in the conformal class of~$\eta$, and~$\EH(\alpha_\varepsilon)=\REHmin$. Up to rescaling~$\chi_{\min}$ by a constant we can also assume that~$\alpha_\varepsilon$ has unitary volume, so that
\begin{equation}
    \int_N\ScalTW(\alpha_\varepsilon)\vol_{\alpha_\varepsilon}=\EH(\alpha_\varepsilon)=\REHmin.
\end{equation}
We then let~$u_\varepsilon$ be the unique~$\TorusL$-invariant solution of
\begin{equation}\label{eq:def_u_epsilon}
    \begin{cases}
    u_\varepsilon\ScalTW(\alpha_\varepsilon)+2q\Delta_{\varepsilon}u_\varepsilon=u_\varepsilon^{q-1}\,\CRYam(\alpha_\varepsilon)\\
    \norm{u_\varepsilon}_{L^q(\alpha_\varepsilon)}=1.
    \end{cases}
\end{equation}
If one can prove that~$u_\varepsilon$ is sufficiently close to~$1$ (in an appropriate norm), it will follow that~$\ScalTW(\alpha_\varepsilon)$ is approximately constant. For example, since~\eqref{eq:def_u_epsilon} implies
\begin{equation}
    \abs*{\ScalTW(\alpha_\varepsilon)-\REHmin}\leq u_\varepsilon^{q-2}\abs*{\CRYamT(\alpha_\varepsilon)-\REHmin}-\REHmin\abs*{u_\varepsilon^{q-2}-1}+2q\abs*{u_\varepsilon^{-1}\Delta_{\alpha_\varepsilon}u_\varepsilon}
\end{equation}
then a bound like~$\norm{u_\varepsilon-1}_{\m{C}^2}<\varepsilon$ would imply~$\abs*{\ScalTW(\alpha_\varepsilon)-\REHmin}\leq C\varepsilon$ for some constant~$C$. We are not yet able to obtain such bounds. However, the two results below give partial evidence that, under our assumptions, it is reasonable to expect that the solution of~\eqref{eq:def_u_epsilon} is close to a constant.

\smallskip

For the remainder of this Section, we assume~$\REHmin\leq 0$ and~$\CRYamInv^{\TorusL}(X,L)=\REHmin$. We let~$\alpha_\varepsilon$ be a sequence of Sasaki structures with Reeb vector field~$\chi_{\min}$ and of unit volume such that~$\CRYamT(\alpha_\varepsilon)\geq\REHmin-\varepsilon$. We also let~$u_\varepsilon$ be the family of~$\TorusL$-invariant functions defined by~\eqref{eq:def_u_epsilon}.

\begin{lemma}\label{lemma:approx_cscK_gradient}
    With the above notation,~$\norm*{\dd\log u_\varepsilon}^2_{L^2(\alpha_\varepsilon)}\leq\varepsilon$.
\end{lemma}
\begin{proof}
    Proposition~\ref{prop:negative_CRYam} gives
    \begin{equation}\label{eq:norm_epsilon_ineq}
        \left(\int_N\abs*{2q\,u_{\varepsilon}^{-1}\Delta_{\alpha_\varepsilon}u_\varepsilon+\ScalTW(\alpha_\varepsilon)}^{n+1}\vol_{\alpha_\varepsilon}\right)^{\frac{1}{n+1}}= -\CRYamT(\alpha_\varepsilon)\leq-\REHmin+\varepsilon.
    \end{equation}
    Jensen's inequality for the function~$f(x)=\abs{x}^{n+1}$ shows that~$\norm{\psi}_{L^1}\leq\norm{\psi}_{L^{n+1}}$ for any function~$\psi$, hence from~\eqref{eq:norm_epsilon_ineq} we obtain 
    \begin{equation}\label{eq:abs_epsilon_ineq}
        \int_N\abs*{2q\,u_{\varepsilon}^{-1}\Delta_{\alpha_\varepsilon}u_\varepsilon+\ScalTW(\alpha_\varepsilon)}\vol_{\alpha_\varepsilon} \leq -\REHmin+\varepsilon.
    \end{equation}
    The cscTW equation~\eqref{eq:def_u_epsilon} and the assumption~$\REHmin\leq 0$ imply
    \begin{equation}
        2q\,u_\varepsilon^{-1}\Delta_{\alpha_\varepsilon} u_\varepsilon+\ScalTW(\alpha_\varepsilon) \leq 0.
    \end{equation}
    Hence we can rewrite the inequality~\eqref{eq:abs_epsilon_ineq} as
    \begin{equation}
    \begin{split}
        -\REHmin+\varepsilon \geq & -\int_N \big( 2q\,u_{\varepsilon}^{-1}\Delta_{\alpha_\varepsilon}u_\varepsilon +\ScalTW(\alpha_\varepsilon)\big)\vol_{\alpha_\varepsilon} \\
        = & -2q\int_N u_{\varepsilon}^{-1}\Delta_{\alpha_\varepsilon}u_\varepsilon \vol_{\alpha_\varepsilon} -\REHmin .
    \end{split}
    \end{equation}
    from which we obtain the desired inequality for the differential of~$u_\varepsilon$:
    \begin{equation}
        \varepsilon \geq - 2q\int_N u_\varepsilon^{-1}\Delta_{\alpha_\varepsilon} u_\varepsilon \vol_{\alpha_\varepsilon} = 2q\int_N u_\varepsilon^{-2}\abs{\dd u_\varepsilon}_{\alpha_\varepsilon}^2 \vol_{\alpha_\varepsilon}.\qedhere
    \end{equation}
\end{proof}
We can also estimate the distance between~$u_\varepsilon$ and its average, assuming that~$\REHmin$ is \emph{strictly} negative.
\begin{lemma}\label{lem:u^k_estimate}
    Under the same assumptions,
    \begin{equation}
        \abs{\REHmin}\, \abs*{\int_N u_\varepsilon^k\vol_{\alpha_\varepsilon} - 1 } \leq \varepsilon
    \end{equation}
    for every~$k\in[\frac{2}{n},2+\frac{2}{n}]$. In particular, if~$n>1$ then~$\abs{\EH_0}\, \norm*{u_\varepsilon-1}_{L^2(\alpha_\varepsilon)} \leq \varepsilon$.
\end{lemma}
\begin{proof}
    The inequality is obviously true if~$\REHmin=0$, so we assume~$\REHmin<0$ in what follows. From~\eqref{eq:def_u_epsilon} we find
    \begin{equation}
        \ScalTW(\alpha_\varepsilon)+2q\,u_\varepsilon^{-1}\Delta_{\alpha_\varepsilon}u_\varepsilon = u_\varepsilon^{q-2}\,\CRYamT(\alpha_\varepsilon) \geq u_\varepsilon^{q-2} (\REHmin-\varepsilon),
    \end{equation}
    and integrating by parts gives
    \begin{equation}
        \REHmin \geq \REHmin + 2q \int_N u_\varepsilon^{-1}\Delta_{\alpha_\varepsilon}u_\varepsilon\,\vol_{\alpha_\varepsilon} \geq (\EH_0-\varepsilon) \int_N u_{\varepsilon}^{q-2} \vol_{\alpha_\varepsilon}.
    \end{equation}
    We obtain an inequality for the~$L^q$-norm, as~$\REHmin$ is negative:
    \begin{equation}
        \abs{\REHmin}\leq\left(\abs{\REHmin}+\varepsilon\right)\int_N u_{\varepsilon}^{q-2}\vol_{\alpha_\varepsilon}.
    \end{equation}
    By Jensen's inequality,~$0\leq 1-\int_N u^{q-2}_\varepsilon \vol_{\alpha_\varepsilon}$ as~$\int_N u^q_\varepsilon \vol_{\alpha_\varepsilon} =1$. On the other hand, we can estimate from above (recall~$\REHmin<0$) as
    \begin{equation}\label{eq:jensen_almosteq}
        0\leq 1-\int_N u^{q-2}_\varepsilon\vol_{\alpha_\varepsilon} \leq 1-\left(1+\frac{\varepsilon}{\abs{\REHmin}}\right)^{-1} \approx \frac{\varepsilon}{\abs{\REHmin}}.
    \end{equation}
    The thesis for~$u^k$ follows from~\eqref{eq:jensen_almosteq} by another application of Jensen's inequality.
\end{proof}

\section{The CR Yamabe invariant and K-stability}\label{sec:Kstab}

In~\cite{LahdiliLegendreScarpa} we considered the Einstein-Hilbert functional on the fibres~$(X_\tau,L_\tau)$ of a test configuration~$(\tstX,\tstL)$ for~$(X,L)$, and we showed that the limit as~$\tau\to 0$ of the Einstein-Hilbert functionals detects the Donaldson-Futaki weight of the test configuration. In this Section we show a generalisation of this result, relating the existence of cscS structures and Sasakian K-stability via the~$\EH$-functional. The main technical results are proven in Section~\ref{sec:Appendix}, adapting some arguments of~\cite{LahdiliLegendreScarpa}.

Throughout this Section, we consider the manifold~$N\coloneqq(L^{-1}\setminus X)/\R_{>0}$ with the transversally complex structure defined by~$X$ and~$\xi_0$, a torus action~$\TorusL\curvearrowright N$ as in \S\ref{sec:torus_CR}, and we fix a section~$\sigma$ of~$\Kmet(X,L)^{\TorusL}\to\chern_1(L)_+^{\TorusL}$ such that~$\sigma(\omega +\ddc\varphi)=\sigma(\omega) +\dc\varphi$. For~$\omega\in\chern_1(L)_+^{\TorusL}$ we let~$\eta_\omega\coloneqq\sigma(\omega)$.

\smallskip

Before coming to the relation between the Einstein-Hilbert functional and K-stability, let us recall the definition of a \emph{weak geodesic ray} of K\"ahler potentials. These are one-parameter paths of functions~$(\varphi_t)_{t\geq 0}\in\PSH(X,\omega)$ such that the~$U(1)$-invariant function on~$X\times\mathbb{D}^*$ defined by
\begin{equation}\label{eq:S1_inv_funct}
    \Phi(x,\tau)\coloneqq\varphi_t(x),\quad \tau=\e^{-t+\I\theta}
\end{equation}
is~$\proj_{1}^*\omega$-plurisubharmonic and solves~$\left(\proj_1^*\omega+\ddc\Phi\right)^{n+1}=0$ in the sense of Bedford-Taylor, where~$\proj_{1}:X\times\mathbb{D}^*\to X$ denotes the projection on the first factor.

\begin{definition}\label{def:ribbon_geod}
    Let~$\chi\in\torusL_+$ be a Reeb vector field. Given a weak geodesic ray~$(\varphi_t)_{t\geq 0} \in \PSH^{\TorusL}(X,\omega)$, the \emph{ribbon of weak geodesics} defined by~$\chi$ and~$\varphi_t$ on~$N$ is the~$2$-parameters path of contact forms~$\eta_{s,t} \coloneqq f_{s,t}^{-1}\eta_{\omega+\ddc\varphi_t}$, for a conformal factor defined by
    \begin{equation}\label{eq:ribbon_conffactor}
        f_{s,t}\coloneqq \langle\chi,\mu_{\omega+\ddc\varphi_t}\rangle+s\dot{\varphi}_t = \eta_{\omega+\ddc\varphi_t}(\chi)+s\dot{\varphi}_t.
    \end{equation}
\end{definition}
Recall that in general~$\varphi_t$ is of~$\mathcal{C}^{1,1}$ regularity as a function on~$X\times\R_+$~\cite{CTW2}, so that there exists~$s$ such that~$f_{s,t}>0$ for every~$t$. Indeed, the derivative~$\dot{\varphi}_t$ is uniformly bounded, and~\eqref{eqMommap=eta} shows that~$\langle\chi,\mu_{\varphi_t}\rangle$ is well-defined as long as~$\dc\varphi_t$ is defined.

\smallskip

Given a K\"ahler form~$\omega\in\chern_1(L)^{\TorusL}_+$, to any smooth, ample, dominant, and~$\TorusL$-equivariant test configuration~$(\tstX,\tstL)$ for~$(X,L)$, there is an associated weak geodesic ray of K\"ahler potentials~$\varphi_t\in\PSH^{\TorusL}(X,\omega)$ starting from~$\omega$, see~\cite{PhongSturm_geodrays,Zak_Ksemistab}, and so we obtain a corresponding ribbon of weak geodesics for any~$\chi\in\TorusL_+$. In this case, the conformal factor~$f_{s,t}$ defined in~\eqref{eq:ribbon_conffactor} can be defined in a more geometric way. We outline this, before stating the main results of this Section.

\paragraph{A Sasaki test configuration.} Let~$(\tstX,\tstL)$ be a test configuration as above for the polarised manifold~$(X,L)$. To any positively curved Hermitian metric on~$\tstL$, the Boothby-Wang construction associates a Sasaki manifold~$(\tstN,\hat{\eta})$ such that~$\dd\hat{\eta}$ descends to a K\"ahler form~$\Omega$ on~$\tstX$, and we interpret~$\tstN$ as a test configuration for the Boothby-Wang Sasaki manifold of~$(X,L)$. Since~$(\tstX,\tstL)$ is~$\TorusL$-equivariant, there is an induced action~$\TorusL\times U(1)\curvearrowright\tstN$, where the extra~$U(1)$-factor is given by the test configuration action of~$(\tstX,\tstL)$, which covers the standard~$U(1)$-action on~$\C$. We denote by~$\zeta$ the generator of this~$U(1)$ action, and for any~$\chi\in\Rcone$ we can consider~$\chi-s\zeta\in\Lie(\TorusL\times U(1))$, which is a Sasaki-Reeb vector field on~$\tstN$. 

Lemma~$4.8$ in~\cite{LahdiliLegendreScarpa} shows that, if the CR contact form~$\hat{\eta}\in\Kmet(\tstX,\tstL)$ is chosen so that~$\Omega$ restricts to~$\omega+\ddc\varphi_t$ on each fibre of the test configuration, then the ribbon of conformal factors~$f_{s,t}$ of~\eqref{eq:ribbon_conffactor} is the pull-back to~$(X,L)\simeq(X_1,L_1)$, along the~$\zeta$-action, of the contact moment map~$\hat{\eta}(\chi-s\zeta)_{\restriction N_t}$.

\subsection{The Einstein-Hilbert functional on test configurations}\label{sec:EH_testconf}

\begin{definition}\label{def:Globalformula}
    Let~$(\tstX,\tstL)$ be a smooth ample dominating test configuration for~$(X,L)$ with reduced central fibre and generating vector field~$\zeta$, and let~$\chi\in\TorusL_+$ be a Reeb vector field.
    Fix~$\hat{\eta}\in\Kmet(\tstX,\tstL)^\TorusL$. For~$s\in \R$ such that~$\hat{\eta}(\chi-s\zeta)>0$ on~$\tstN$, we let~$\hat{\eta}_s\coloneqq\hat{\eta}(\chi-s\zeta)^{-1}\hat{\eta}$ and
    \begin{equation}\label{eq:EHinv}
    \begin{split}
        \bfS^\chi_s(\tstX,\tstL) \coloneqq &  \frac{1}{(1-s\mu^{\zeta}_{\chi,\max})^n}  \Scaltot(N,\chi)  - ns\, \Scaltot(\tstN, \chi-s\zeta) \\
        &+2ns \int_{{\tstN}} \pi^*\omega_{\FS} \wedge \hat{\eta}_s \wedge  (\dd\hat{\eta}_s)^{[n]}  +2n(n+1)s^2\int_{{\tstN}} \pi^*\mu_{\FS}\, \hat{\eta}_s \wedge (\dd\hat{\eta}_s)^{[n+1]}\\
        \bfV_s^\chi(\tstX,\tstL) \coloneqq & \frac{1}{(1-s\mu^{\zeta}_{\chi,\max})^{(n+1)}}\Voltot(N,\chi) -(n+1)s \Voltot(\tstN, \chi-s\zeta)
    \end{split}
    \end{equation}
    where~$\mu_{\FS}$ is the moment map of the standard~$U(1)$-action on~$\mathbb{P}^1$ with respect to the Fubini-Study form, and~$\mu^{\zeta}_{\chi,\max}$ is the maximal value reached by the moment map~$\mu^{\zeta}_{\chi}=\hat{\eta}(\chi)^{-1}\hat{\eta}(\zeta)$ on~$\tstX$ (equivalently on~$\tstN$).
\end{definition}
Note that the second line of~$\bfS^\chi_s(\tstX,\tstL)$ can also be written as the term~$B^{\pi^*\omega_{\FS}}_{\mathrm{v}}(\Kahmap(\hat{\eta}))$ in~\cite[Lemma 2]{Lahdili_weighted} and therefore does not depend on the chosen metric.

\smallskip

The main result of this section is the following generalisation of~\cite[Theorem~$1.4$]{LahdiliLegendreScarpa}.
\begin{theorem}\label{thm:EH_limit_testconf}
    Let~$(\tstX,\tstL)$ be a smooth ample dominating test configuration for~$(X,L)$ with reduced central fibre and generating vector field~$\zeta$, and let~$\chi\in\TorusL_+$ be a Reeb vector field. For any~$\omega\in\chern_1(L)_+^{\TorusL}$ and~$\abs{s}\ll 1$ consider the ribbon of weak geodesics~$\eta_{s,t}$ defined by~$\chi$ and the weak geodesic ray associated to~$\omega$ and~$(\tstX,\tstL)$. Then,
    \begin{equation}\label{eq:Scaltot_tst}
        \lim_{t\to+\infty}\Scaltot(\eta_{s,t})=\bfS_s^\chi(\tstX,\tstL)\
        \text{ and }
        \lim_{t\to+\infty}\Voltot(\eta_{s,t})=\bfV_s^\chi(\tstX,\tstL).
    \end{equation}
    In particular,
    \begin{equation}\label{eq:EH_tst}
        \lim_{t\to\infty}\EH(\eta_{s,t}) = \EH^\chi_s(\tstX,\tstL) \coloneqq  \frac{\bfS^\chi_s(\tstX,\tstL)}{\bfV^\chi_s(\tstX,\tstL)^{\frac{n}{n+1}}}
    \end{equation}
    and these limits do not depend on the initial choice of~$\omega$.
\end{theorem}
We think of~$\EH^{\chi}_s(\tstX,\tstL)$ as the Einstein-Hilbert functional~$\REH$ of the central fibre, evaluated on the Reeb vector field~$\chi-s\zeta$. Note that~$\EH$ (more precisely, the total scalar curvature functional) may be not well-defined along the ribbon, as the potentials~$\varphi_t$ only have~$\m{C}^{1,1}$ regularity, in general, so that~$\Scal(\omega_t)$ is not well-defined. To make sense of the limit in Theorem~\ref{thm:EH_limit_testconf} it is necessary to first extend~$\EH$, along a ribbon of weak geodesics. This was done in~\cite{LahdiliLegendreScarpa} for the particular choice~$\chi=\xi_0$ by introducing an \emph{action functional}~$\Action_s(t)$, satisfying
\begin{equation}\label{eq:actionfunctional_def}
    \Action(s,t)=\int_0^t\Scaltot(\eta_{s,x})\dd x
\end{equation}
whenever~$\Scaltot$ is well-defined along the ribbon~$\eta_{s,t}$.
It turns out that~$\Action(s,t)$ can be extended as a functional on the set of ribbons of~$\m{C}^{1,1}$ regularity, see~\cite[Corollary~$4.18$]{LahdiliLegendreScarpa}. It is immediate to generalise this to the case of a general Reeb vector field~$\chi$, giving a functional~$\Action_\chi(s,t)$ satisfying~$\partial_t\Action_\chi(s,t)=\Scaltot(\eta_{s,t})$. With this in mind, the limit~\eqref{eq:Scaltot_tst} is more precisely stated as
\begin{equation}\label{eq:Act_chi_lim_testconf}
    \bfS^\chi_s(\tstX,\tstL)=\lim_{t\to+\infty}\partial_t\Action_\chi(s,t).
\end{equation}

Taking the first terms of the expansion in~$s$ of~$\bfS^\chi_s$ and~$\bfV^\chi_s$ one obtains an important consequence of Theorem~\ref{thm:EH_limit_testconf}.
\begin{lemma}\label{lem:derivEH=DF} 
    Under the hypothesis of Definition~\ref{def:Globalformula}, we have
    \begin{equation}
        \frac{\dd}{\dd s}\Bigr|_{s=0}\EH^\chi_s(\tstX,\tstL) = \frac{2n}{\Voltot(N,\chi)^{\frac{n}{n+1}}} \SasFut(\tstX,\tstL,\chi,\zeta)
    \end{equation}
    where~$\SasFut(\tstX,\tstL,\chi,\zeta)$ is the \emph{global Sasaki-Futaki invariant} of~\cite{ApostolovCalderbankLegendre_weighted}.  
\end{lemma}
We recall from~\cite[Definition~$6.7$]{ApostolovCalderbankLegendre_weighted} that, up to a positive dimensional constant, 
\begin{equation}\label{eqSASFUT}
\begin{split}
    \SasFut(\tstX,\tstL,\chi,\zeta) = &  \Voltot(\tstN, \chi-s\zeta)\frac{ \Scaltot(N,\chi)}{\Voltot(N,\chi)}- \Scaltot(\tstN, \chi-s\zeta)\\
    &  +2\int_{{\tstN}} \pi^*\omega_{\FS} \wedge \hat{\eta}_s \wedge  (\dd\hat{\eta}_s)^{[n-1]}.
\end{split}
\end{equation}

\begin{remark}
    It was proved in~\cite[Proposition 6.8]{ApostolovCalderbankLegendre_weighted}, by approximation by quasi-regular Sasaki-Reeb vector fields, that for a smooth ample test configuration, the invariant~\eqref{eqSASFUT} coincides with the algebraic Sasaki-Donaldson-Futaki invariant introduced by Collins-Sz\'ekelyhidi~\cite{CollinsSzekelyhidi_KsemistabSasaki}.       
\end{remark}

It is easy to verify that if~$\eta_{s,t}$ is a \emph{smooth} ribbon of CR contact forms then
\begin{equation}
    \partial_t\Scaltot(\eta_{s,t})=sn\int_N\left(\abs*{\nablaminus\dot{\varphi}_t}^2_{\omega_{\varphi_t}} -f_{s,t}^{-1}\Scaltot(\eta_{s,t})\left(\ddot{\varphi}_t-\frac{1}{2}\abs{\dd\dot{\varphi}_t}^2_{\omega_{\varphi_t}}\right)\right)\vol_{\eta_{s,t}}
\end{equation}
so that for~$s>0$,~$\Scaltot(\eta_{s,t})$ is an increasing function of~$t$ if~$\eta_{s,t}$ is a smooth geodesic ribbon. The most important technical result that leads to Theorem~\ref{thm:EH_limit_testconf} is a generalisation of this property of~$\Scaltot(\eta_{s,t})$ to weak geodesic rays, using the action functional.
\begin{theorem}[c.f.\ Theorem~$1.4$ in~\cite{LahdiliLegendreScarpa}]\label{theo:ActionPointwiseCvxC0}
    With the notation of Theorem~\ref{thm:EH_limit_testconf}, the function~$t\mapsto \Action_\chi(s,t)$ is pointwise convex and continuous on~$[0,\infty)$ for~$s>0$.
\end{theorem}
Another important property of the action functional is a \emph{slope inequality}, relating it derivative at~$t=0$ with the total scalar curvature.
\begin{proposition}\label{prop:A>EH}
    With the previous notation,
    \begin{equation}
         \frac{\partial}{\partial t}_{|0^+}\Action_\chi (s,t) \geq \Scaltot(\eta_{s,0}),
    \end{equation}
    and equality holds for smooth geodesic ribbons.
\end{proposition}
With these results we can show a refinement of the inequality~$\CRYamInv^{\TorusL}(X,L)\leq\REHmin$.
\begin{corollary}\label{cor:CRYam_ineq}
    Let~$(\tstX,\tstL)$ be a smooth ample dominant~$\TorusL$-equivariant test configuration for~$(X,L)$. Then, for any~$\chi\in\Rcone$ and any~$0<s\ll 1$,
    \begin{equation}
        \CRYamT(X,L)\leq\EH^\chi_s(\tstX,\tstL).
    \end{equation}
\end{corollary}

\begin{proof}
    Fix~$\omega\in\chern_1(L)_+^{\TorusL}$, and let~$\eta_{s,t}$ be the ribbon of weak geodesics corresponding to the weak geodesic ray of K\"ahler potentials starting from~$\omega$ associated to~$(\tstX,\tstL)$ and~$\chi$ as in~\eqref{eq:ribbon_conffactor}. We consider the action functional along the ribbon, and the limit~\eqref{eq:Act_chi_lim_testconf} for~$s>0$ small enough for~$f_{s,t}$ to be positive.
    
    We now use the convexity of Theorem~\ref{theo:ActionPointwiseCvxC0}, Proposition~\ref{prop:A>EH}, and the fact that the volume functional is constant (in~$t$) along a ribbon of weak geodesics (see Lemma~\ref{lem:Const-geodesic} below), to obtain
    \begin{equation}
        \EH^\chi_s(\tstX,\tstL) = \frac{\partial_{t}\Action_\chi(s,t)}{\Voltot(\eta_{s,t})^{\frac{n}{n+1}}} \geq \frac{\partial_{t=0}\Action_\chi(s,t)}{\Voltot(\eta_{s,0})^{\frac{n}{n+1}}} \geq \EH(\eta_{s,0}).
    \end{equation}
    On the other hand by Proposition~\ref{prop:CRYam_solution}, as~$f_{s,0}\in\m{C}^{1,1}(X)^{\TorusL}\subset W^{1,2}(X,\omega)^{\TorusL}$,
    \begin{equation}
        \EH(\eta_{s,0})=\EH(f_{s,0}^{-1}\eta_\omega) \geq \inf_{f\in\m{C}^{1,1}(X,\R_{>0})^{\TorusL}}\EH(f^{-1}\eta_\omega)=\CRYamT(\eta_\omega).
    \end{equation}
    Putting all of this together shows that, for every~$0<s\ll 1$,
    \begin{equation}
        \EH^\chi_s(\tstX,\tstL) \geq \CRYam(\eta_\omega),
    \end{equation}
    and this gives the thesis, as~$\EH^\chi_s(\tstX,\tstL)$ does not depend on the choice of the initial K\"ahler form~$\omega$.
\end{proof}
    
Theorem~\ref{thm:Ksemistable} is now a simple consequence of this inequality between~$\CRYamT(X,L)$ and the Einstein-Hilbert functional of test configurations and Theorem~\ref{thm:EH_limit_testconf}.
\begin{proof}[Proof of Theorem~\ref{thm:Ksemistable}]
    If~$\CRYamT(X,L)=\REHmin$ and we choose~$\chi=\chi_{\min}$ in the previous Corollary, we obtain
    $$\EH^{\chi_{\min}}_s(\tstX,\tstL)-\REHmin\geq 0$$ for every small enough~$s>0$, which implies~$\SasFut(\tstX,\tstL,\chi,\zeta)\geq 0$ by Lemma~\ref{lem:derivEH=DF}.
\end{proof}
The proof of Corollary~\ref{cor:CRYam_ineq} can also be adapted to show that if there exists a Sasaki structure with constant (transversal) scalar curvature and Reeb vector field~$\chi$, then the Sasaki manifold ~$(N,\chi)$ is K-semistable, c.f.\ \cite[Corollary~$1.6$]{LahdiliLegendreScarpa}.

In the case when~$\chi\in\Rcone$ is \emph{regular}, the results of this section (Theorem~\ref{thm:EH_limit_testconf}, Theorem~\ref{theo:ActionPointwiseCvxC0}, Proposition~\ref{prop:A>EH}) follow immediately from the results of~\cite[\S$4$]{LahdiliLegendreScarpa}. To see this, consider any~$\eta\in\Kmet(\tstX,\tstL)$. This is a Sasaki structure with Reeb vector field~$\xi_0$, the generator of the~$U(1)$-action on~$L$, and~$\alpha=\eta(\chi)^{-1}\chi\in\Cmet(\tstX,\tstL)$ is a second Sasaki structure with Reeb vector field~$\chi$. Assuming that~$\chi$ is regular, we obtain a new line bundle structure~$\tstL\to\tstX'$, with fibres given by the complexified orbits of~$\chi$. Now, note that Remark~\ref{rmk:propostional_sasaki} shows that
\begin{equation}
    \eta_{s,t}= f_{s,t}^{-1}\eta_t=\iota_t^*\left(\eta(\chi-s\zeta)^{-1}\eta\right)= \iota_t^*\left(\alpha(\chi-s\zeta)^{-1}\alpha\right),
\end{equation}
and~$\chi$ is the Reeb vector field of~$\alpha$, so we can directly apply the results of~\cite[\S$4$]{LahdiliLegendreScarpa}.

As in general we can not assume~$\chi$ to be regular, we prove these results in \S\ref{sec:Appendix} by computing on~$X$, the quotient under the~$\xi_0$-action, rather than the possibly non-existent quotient by~$\chi$, which shows why we assume that there exists at least one regular Reeb vector field. We expect however these results to hold in the more general case when no regular fields are present.

\subsection{The action functional and test configurations}\label{sec:Appendix}

We start by briefly recalling the construction of the geodesic ray associated to a (dominant, smooth, ample) test configuration~$(\tstX,\tstL)$ starting from a K\"ahler form~$\omega\in\chern_1(L)_+$, from~\cite{PhongSturm_geodrays,Zak_Ksemistab}. We take into account a possible torus action on~$L$.

As the test configuration is dominant, we have a ($\C^*\times\TorusL$-equivariant) bimeromorphic morphism~$\Pi:\tstX\to X\times\mathbb{P}^1$, and we have~$\tstL=\Pi^* {\proj_1}^* L+D$ for a unique~$\mathbb{Q}$-Cartier divisor~$D$ supported on the central fiber~${\tstX}_0$, see~\cite{Zak_Ksemistab}.

By the Poincaré-Lelong formula, the current of integration~$\delta_D$ of~$D$ can be written as~$\delta_D=\ddc \gamma_D+\Theta_D$, where~$\Theta_D$ is a smooth~$\TorusL\times\mathbb{S}^1$-invariant representative of the fundamental class~$[D]$ of~$D$, and~$\gamma_D$ is a~$U(1)\times\TorusL$-invariant Green's function on~$\tstX$ defined outside the central fibre.

We can assume that~$\Omega=\Pi^*{\proj_1}^*\omega+\Theta_D\in \chern_1(\tstL)^{U(1)\times\TorusL}_+$ is a K\"ahler form on~$\tstX$. By~\cite[Theorem 1.2]{CTW2}, on the compact manifold with boundary~${\tstX}_{\mathbb{D}}\coloneqq\TCmap^{-1}(\mathbb{D})$, the boundary value problem
\begin{equation}\label{BVP:G}
    \begin{cases}
        (\Omega+\ddc\Gamma)^{n+1}=0\\
        \Gamma_{|\partial {\tstX}_{\mathbb{D}}}=\varphi_0+\gamma_D
    \end{cases}    
\end{equation}
admits a unique solution~$\Gamma\in\m{C}^{1,1}({\tstX}_{\mathbb{D}})$, that will be~$U(1)\times\TorusL$-invariant as both~$\Omega$ and the boundary data are. The weak geodesic ray~$\Phi_t$ is the unique function on~$X\times\mathbb{D}^*$ such that~$\Pi^*\Phi\coloneqq\Gamma-\gamma_D$. It is easy to check then that~$\Pi^*({\proj}_1^*\omega+\ddc\Phi)=\Omega+\ddc\Gamma-\delta_D$, and
\begin{equation}\label{eq:weakgeod_testconf}
    \begin{cases}
        (\Omega+\ddc\Phi)^{n+1}=0 \text{ on }X\times\mathbb{D}^*\\
        \Phi_{\restriction\set{\abs{\tau}=1}}=\varphi_0.
    \end{cases} 
\end{equation}

\subsubsection{The action functional on ribbons of weak geodesics}

In this and the next subsections, we sketch the proof of Theorem~\ref{thm:EH_limit_testconf} and Theorem~\ref{theo:ActionPointwiseCvxC0}. We refer to~\cite[\S$4.3$ and \S$4.4$]{LahdiliLegendreScarpa}, in which these results were proven for the case~$\chi=\xi_0$ (the generator of the~$U(1)$-action on the fibres of~$L\to X$), for more details.

\smallskip

We start by stating a small result that will be used repeatedly in what follows, which is checked by a simple calculation, see also~\cite[\S$1.1$]{FutakiMabuchi_extremal}.
\begin{lemma}\label{lem:Const-geodesic}
    Let~$F(\mu,x)$ be a smooth real-valued function on~$\Pol_L\times\R$, and let~$(\varphi_t)_{t\geq 0}$ be a~$\TorusL$-invariant weak geodesic ray on~$(X,\omega)$. Then, $$\int_XF(\mu_{\varphi_t},\dot{\varphi}_t)\omega_{\varphi_t}^{[n]}$$ is a constant independent of~$t$.
\end{lemma}
In what follows,~$(\tstX,\tstL)$ will always denote a~$\TorusL$-equivariant smooth ample dominant test configuration for~$(X,L)$ with reduced central fibre and generating vector field~$\zeta$, and~$\chi\in\TorusL_+$ will be a fixed Reeb field. We also fix~$\Omega\in\chern_1(\tstL)_+^{\TorusTest}$ and let~$\hat{\eta}\in\Kmet(\tstX,\tstL)^{\TorusTest}$ be a connection~$1$-form with curvature~$\Omega$. The form~$\Omega$ induces a K\"ahler form~$\omega$ on~$(X_1,L_1)\simeq(X,L)$, and we let~$\varphi_t$ be the (sub)geodesic ray starting from~$\omega$ corresponding to the test configuration.

The following is a key computation to obtain an explicit expression for~$\Action_\chi(s,t)$. It is a slight generalisation of~\cite[\S$4.17$]{LahdiliLegendreScarpa}, and we follow the same computations.
\begin{lemma}\label{lem:chS-eq}
    Let~$\varphi_t$ be a subgeodesic ray,~$f_{s,t}=\langle\chi,\mu_{\varphi_t}\rangle+s\dot{\varphi}_t$, and~$\eta_{s,t}\coloneqq f_{s,t}^{-1}\eta_{\omega_{\varphi_t}}$.
    Let also~$\mathbf{h}(x,\tau)$ be the ~$\TorusL\times U(1)$-invariant function induced on~$X\times\mathbb{D}^*$ by 
    \begin{equation}\label{eq:psi_def}
        h_t \coloneqq \log\left(\langle\chi,\mu_{\varphi_t}\rangle^{-n-1}\frac{\omega_{\varphi_t}^{[n]}}{\omega^{[n]}}\right),
    \end{equation}
    and let~$\F_{s,\tau}$ be the~$U(1)$-invariant conformal factor on~$X\times\mathbb{D}^*$ corresponding to~$f_{s,t}$.
    Then, the total Tanaka-Webster scalar curvature~$\Scaltot(\eta_{s,t})$ admits the following expression
    \begin{equation}
    \begin{split}
        \Scaltot(\eta_{s,t})=& ns\frac{d}{dt} \left( \int_X h_t  \frac{\omega_{\varphi_t}^{[n]}}{f_{s,t}^{n+1}}
        -(n+1)s\int_0^t\int_{X} \mathbf{h}\frac{\partial_\theta \contract  (\omega+\ddc\Phi)^{[n+1]}}{\F_{s,\tau}^{n+2}} \right) \\
        &+ 2\int_X \left( \Ric(\omega)\wedge \frac{\omega_{\varphi_t}^{[n-1]}}{f_{s,t}^n} -n\langle \mu_{\Ric(\omega)},\chi\rangle \frac{\omega^{[n]}_{\varphi_t}}{f_{s,t}^{n+1}} \right),
    \end{split}
    \end{equation}
    where~$\partial_\theta$ is the generator of the standard~$\mathbb{S}^1$-action on~$\mathbb{D}^{*}$,~$\mu_{\Ric(\omega)}=\frac{1}{2}\Delta\mu_\omega$.
\end{lemma}
There is a slight abuse of notation in this formula, as~$\Scaltot(\eta_{s,t})$ should be written as an integral over~$N$ rather than~$X$. As all quantities are~$\TorusL$-invariant however, the two expressions differ only by a factor of~$2\pi$, the volume of a fibre of~$N\to X$.
\begin{proof}
    We start from the total scalar curvature of~$\eta_{s,t}$,
    \begin{equation}
        \Scaltot(\eta_{s,t})=\int_X f_{s,t}^{-n}\left(\Scal(\omega_{\varphi_t})+n(n+1)f_{s,t}^{-2}\abs*{\dd f_{s,t}}^2_{\varphi_t}\right)\omega_{\varphi_t}^{[n]}.
    \end{equation}
    Plugging the identity
    \begin{equation}
        \Scal(\omega_{\varphi_t})\omega_{\varphi_t}^{[n]} = 2 \Ric(\omega)\wedge \omega_{\varphi_t}^{[n-1]}  -\ddc\log\left(\frac{\omega_{\varphi_t}^{[n]}}{\omega^{[n]}}\right)\wedge\omega_{\varphi_t}^{[n-1]}
    \end{equation}
    in the definition of~$\Scaltot(\eta_{s,t})$ and integrating the Laplacian term by parts yields
    \begin{equation}\label{eq:wbfS:1}
    \begin{split}
        \Scaltot(\eta_{s,t}) =& 
        \int_X\left(
            2\,f_{s,t}^{-n} \Ric(\omega)
            +n(n+1)f_{s,t}^{-n-2} \dd f_{s,t} \wedge \dc f_{s,t} \right)
        \wedge \omega_{\varphi_t}^{[n-1]} \\
        &-\int_X n\,f_{s,t}^{-n-1} \dd\log\left(\frac{\omega_{\varphi_t}^{[n]}}{\omega^{[n]}}\right) \wedge \dc f_{s,t} \wedge \omega_{\varphi_t}^{[n-1]}.
    \end{split}
    \end{equation}
    Recall now that~$f_{s,t}=\langle\chi,\mu_\varphi\rangle+s\dot{\varphi}$. Then we can use the relation
    \begin{equation}
        \frac{1}{2}g_{\varphi_t}\left(\dd\log\left(\frac{\omega_{\varphi_t}^{[n]}}{\omega^{[n]}}\right),\dd\langle\chi,\mu_\varphi\rangle\right) = -\langle\mu_{\Ric(\omega_{\varphi_t})},\chi\rangle + \langle \mu_{\Ric(\omega)},\chi\rangle
    \end{equation}
    to rewrite the third term of the right hand side of~\eqref{eq:wbfS:1}, obtaining
    \begin{equation}\label{eq:wbfS:2}
    \begin{split}
        \Scaltot(\eta_{s,t}) = & 2 \int_X \left( f_{s,t}^{-n} \Ric(\omega)\wedge \omega_{\varphi_t}^{[n-1]} - f_{s,t}^{-n-1}  \langle\mu_{\Ric(\omega)},\chi\rangle \omega_{\varphi_t}^{[n]} \right)   \\
        & + \int_X n(n+1)f_{s,t}^{-n-2} \dd f_{s,t} \wedge \dc f_{s,t} \wedge \omega_{\varphi_t}^{[n-1]} \\
        &-\int_X sn\,f_{s,t}^{-n-1} \dd\log\left(\frac{\omega_{\varphi_t}^{[n]}}{\omega^{[n]}}\right) \wedge \dc\dot{\varphi}_t \wedge \omega_{\varphi_t}^{[n-1]}\\
        &+\int_X 2n\,f_{s,t}^{-n-1} \langle\mu_{\Ric(\omega_{\varphi_t})},\chi\rangle \,\omega_{\varphi_t}^{[n]}.
    \end{split}
    \end{equation}
    At this point we substitute~$\mu_{\Ric(\omega_{\varphi_t})}=\frac{1}{2}\Delta_{\varphi_t}(\mu_{\varphi_t})$ and integrate by parts the last term in~\eqref{eq:wbfS:2}. This becomes
    \begin{equation}
    \begin{split}
        \int_X 2n\,f_{s,t}^{-n-1} \langle\mu_{\Ric(\omega_{\varphi_t})},\chi\rangle \,\omega_{\varphi_t}^{[n]}=&
        \int_X n \dd\langle\mu_{\varphi_t},\chi\rangle \wedge \dc f_{s,t}^{-n-1}\wedge\omega_{\varphi_t}^{[n-1]}\\
        =& -\int_X n(n+1)f_{s,t}^{-n-2} \dd (f_{s,t}-s\dot{\varphi}_t) \wedge \dc f_{s,t}\wedge\omega_{\varphi_t}^{[n-1]},
    \end{split}
    \end{equation}
    and putting everything together we get
    \begin{equation}
    \begin{split}
        \Scaltot(\eta_{s,t}) = & 2 \int_X \left( f_{s,t}^{-n} \Ric(\omega)\wedge \omega_{\varphi_t}^{[n-1]} - f_{s,t}^{-n-1}  \langle\mu_{\Ric(\omega)},\chi\rangle \omega_{\varphi_t}^{[n]} \right)\\
        &-sn\int_X \left( \dd f_{s,t}^{-n-1} \wedge \dc \dot{\varphi}_t  + f_{s,t}^{-n-1} \dd\log\left(\frac{\omega_{\varphi_t}^{[n]}}{\omega^{[n]}}\right) \wedge \dc\dot{\varphi}_t \right)\wedge \omega_{\varphi_t}^{[n-1]}
    \end{split}
    \end{equation}
    We just need to find a primitive for the second integral.
    Integrating by parts one gets
    \begin{equation}
    \begin{split}
        \partial_t\left(\int_X h_t \frac{\omega_{\varphi_t}^{[n]}}{f_{s,t}^{n+1}}\right) &
        = - \int_X \left( \dd\dot{\varphi}_t \wedge \dc f_{s,t}^{-n-1} + f_{s,t}^{-n-1} \dd\log\left(\frac{\omega_{\varphi_t}^{n}}{\omega^{n}}\right)\wedge\dc\dot{\varphi}_t\right) \wedge \omega_{\varphi_t}^{[n-1]} \\
        & \quad - (n+1)s \int_X   f_{s,t}^{-n-2} h_t \left( \ddot{\varphi}_{t} - \abs{\dd\dot{\varphi}_t}^2_{\varphi_t} \right) \wedge\omega_{\varphi_t}^{[n-1]}.
    \end{split}
    \end{equation}
    We can use this expression to rewrite the second integral in~\eqref{eq:wbfS:2} as
    \begin{equation}
    \begin{split}
         -\int_X & \left( \dd f_{s,t}^{-n-1} \wedge \dc \dot{\varphi}_t  + f_{s,t}^{-n-1} \dd\log\left(\frac{\omega_{\varphi_t}^{[n]}}{\omega^{[n]}}\right) \wedge \dc\dot{\varphi}_t \right)\wedge \omega_{\varphi_t}^{[n-1]} = \\
        & = \partial_t\left(\int_X h_t \frac{\omega_{\varphi_t}^{[n]}}{f_{s,t}^{n+1}}
        + (n+1)s \int_0^t \dd t \int_X   f_{s,t}^{-n-2} h_t \left( \ddot{\varphi}_{t} - \abs{\dd\dot{\varphi}_t}^2_{\varphi_t} \right) \wedge\omega_{\varphi_t}^{[n-1]}\right)
    \end{split}
    \end{equation}
    The thesis then follows by recalling that 
   ~$(\omega+\ddc\Phi)^{[n+1]}=-(\ddot{\varphi}_{u} -\abs{\dd\dot{\varphi}_{u}}^2_{\varphi_u})\omega_{\varphi_u}^{[n]}\wedge du\wedge d\theta$ on~$X\times \mathbb{D}^*$, so that the last integral becomes
    \begin{equation}
        (n+1)s\int_0^t\int_{X} \mathbf{h}\frac{\partial_\theta \contract  (\omega+\ddc\Phi)^{[n+1]}}{\F_{s,\tau}^{n+2}}.
        \qedhere
    \end{equation}
\end{proof}
Lemma~\ref{lem:chS-eq} gives us an alternative expression for the action functional~\eqref{eq:actionfunctional_def}.
\begin{equation}\label{eq:Action_chi}
\begin{split}
    \Action_{\chi}(s,t) \coloneqq & ns \int_X h_t \frac{\omega_{\varphi_t}^{[n]}}{f_{s,t}^{n+1}}-n(n+1)s^2\int_0^t\int_{X} \mathbf{h}\frac{\partial_\theta \contract  (\omega+\ddc\Phi)^{[n+1]}}{\F_{s,\tau}^{n+2}}\\
    &+ 2\int_0^t\int_X \left( \Ric(\omega)\wedge \frac{\omega_{\varphi_u}^{[n-1]}}{f_{s,u}^n} -n\langle \mu_{\Ric(\omega)},\chi\rangle \frac{\omega^{[n]}_{\varphi_u}}{f_{s,u}^{n+1}} \right)\dd u.
\end{split}
\end{equation} 
where~$h_t$,~$f_{s,t}$, and their corresponding~$U(1)$-invariant functions~$\mathbf{h}$,~$\F_{s,\tau}$ are all defined in Lemma~\ref{lem:chS-eq}.

Note that the second term in~\eqref{eq:Action_chi} vanishes along~$\Torus$-invariant weak geodesic rays. The first term in~\eqref{eq:Action_chi} instead involves the entropy of the measure~$f_{s,t}^{-n-1}\omega_{\varphi_t}^{[n]}$ with respect to~$\omega^{[n]}$. The lower semi-continuity of the entropy~\cite{BermanBerndtsson_uniqueness} then shows that~$\Action_\chi(s,t)$ is well defined and lower semi-continuous as a function of~$t\in[0,\infty)$. The slope inequality along weak geodesics, Proposition~\ref{prop:A>EH}, follows from~\eqref{eq:Action_chi} as in~\cite[Proof of Proposition~$4.15$]{LahdiliLegendreScarpa}, using the positivity of the entropy on the space of probability measures.

\smallbreak

Let~$\Theta_\tau$ be a~$U(1)$-invariant family of volume forms on~$X$, and denote by~$\Psi=(\psi_\tau)_{\tau\in\mathbb{D}^*}$ the corresponding metric on the relative canonical bundle of~$X\times\mathbb{D}^*\to\mathbb{D}^*$. We consider the~$\Psi$-action functional
\begin{equation}
\begin{split}\label{eq:A_Psi-general}
    \Action^\Psi_{\chi}(s,\tau) \coloneqq &
    ns \int_X \log\left( \langle\chi,\mu_\varphi\rangle^{-n-1} \frac{\e^{\psi_\tau}}{\omega^{[n]}}\right) \frac{\omega_{\varphi_t}^{[n]}}{f_{s,t}^{n+1}}\\
    &-n(n+1)s^2\int_0^t\int_{X} \log\left( \langle\chi,\mu_\varphi\rangle^{-n-1}\frac{\e^{\Psi}}{\omega^{[n]}}\right) \frac{\partial_\theta \contract  (\omega+\ddc\Phi)^{[n+1]}}{\F_{s,\tau}^{n+2}}\\
    &+ 2\int_0^t\int_X \left( \Ric(\omega)\wedge \frac{\omega_{\varphi_u}^{[n-1]}}{f_{s,u}^n} -n\langle \mu_{\Ric(\omega)},\chi\rangle \frac{\omega^{[n]}_{\varphi_u}}{f_{s,u}^{n+1}} \right)\dd u.
\end{split}
\end{equation}
This modified action functional coincides with~\eqref{eq:Action_chi} if~$\Psi$ is defined by the metric~$\omega_{\varphi_t}^{[n]}$ on~$K_{X\times\mathbb{D}^*/\mathbb{D}^*}$. However, this function is not locally bounded, which causes additional difficulties with the computations. To avoid this issue, we will first prove the desired results, such as convexity, for the modified action functional~$\Action^\Psi$, and then translate this into properties of~$\Action$ by taking the limit~$\Action^{\Psi_j}_{\chi}$ for a sequence of smooth~$\Psi_j$ approximating the metric induced by~$\omega_{\varphi_t}^{[n]}$. This approximation argument is exactly the same as in~\cite{Inoue_PerelmanEntropy,LahdiliLegendreScarpa}, so we freely use results proved for~$\Action_\chi^\Psi$, such as Proposition~\ref{prop:ddc_A^Psi-smooth} below, for~$\Action_\chi$ without further comment.
\begin{proposition}\label{prop:ddc_A^Psi-smooth}
    If~$\varphi_t$ is a smooth path of K\"ahler potentials and~$\Theta_\tau$ is a smooth family of volume forms we have, in the weak sense of currents,
    \begin{equation}
    \begin{split}
         \ddc \Action^\Psi_{\chi}(s,\tau) =& ns \int_{X} \left( \frac{\ddc\Psi\wedge (\omega+\ddc\Phi)^{[n]} }{\F_{s,\tau}^{n+1}}+(n+1)(\dc\Psi)(s\partial_\vartheta -\chi)\frac{(\omega+\ddc\Phi)^{[n+1]}}{\F_{s,\tau}^{n+2}} \right)
    \end{split}
    \end{equation}
\end{proposition}
Theorem~\ref{theo:ActionPointwiseCvxC0} is a direct corollary of this expression for the second derivative of~$\Action_\chi$ since~$(\omega+\ddc\Phi)^{n+1}=0$ for a weak geodesic ray, see for example~\cite[Theorem~$4.13$]{LahdiliLegendreScarpa}.

\begin{proof}[Proof of Proposition~\ref{prop:ddc_A^Psi-smooth}]
    We compute the differential of~$\Action^\Psi_\chi(s,\tau)$ by integrating against test functions~$\gamma(\tau)$ on~$\mathbb{D}^*$. As all components of~$\Action^\Psi_\chi(s,\tau)$ are~$\TorusL\times U(1)$-invariant, integration by parts shows that~$\int_{X\times\mathbb{D}^*}\Action^\Psi_{\chi}(s,\tau) \ddc\gamma$ is the sum of the three terms
    \begin{equation}\label{eq:ddc-APsi1}
        ns \int_{X\times\mathbb{D}^*} \dc\gamma\wedge \dd\left[ \log\left(\langle\chi,\mu_\Phi\rangle^{-n-1} \frac{\e^{\Psi}}{\omega^{[n]}}\right) \frac{ (\omega+\ddc\Phi)^{[n]} }{\F_{s,t}^{n+1}}\right]
    \end{equation}
    \begin{equation}\label{eq:ddc-APsi2}
        -n(n+1)s^2 \int_{\mathbb{D}^*}\dc\gamma\wedge \left[\int_{X} \log\left(\langle\chi,\mu_\Phi\rangle^{-n-1}\frac{\e^{\Psi}}{\omega^{[n]}}\right) \frac{\partial_\theta \contract  (\omega+\ddc\Phi)^{[n+1]}}{\F_{s,t}^{n+2}}\right]
    \end{equation}
    \begin{equation}\label{eq:ddc-APsi3}
        - 2 \int_{\mathbb{D}^*}\gamma\left[\partial_t\int_X \left( \Ric(\omega)\wedge\frac{\omega_{\varphi_t}^{[n-1]}}{f_{s,t}^n} -n\langle \mu_{\Ric(\omega)},\chi\rangle\frac{ \omega^{[n]}_{\varphi_t}}{f_{s,t}^{n+1}} \right)\right] \dd t \wedge\dd\vartheta
    \end{equation}
    To simplify these, recall that with respect to the coordinate~$\tau=\e^{-t + \I \vartheta}$ we have
    \begin{equation}\label{eq:Om^n}
        (\omega+\ddc\Phi)^{[n+1]}=-\left(\ddot{\varphi}_t-|d\dot{\varphi}_t|^2_{\varphi_t}\right)\omega_{\varphi_t}^{[n]}\wedge \dd t \wedge\dd\vartheta,
    \end{equation}
    \begin{equation}\label{eq:dtheta_contraction}
        \partial_\theta\contract ({\proj_1}^*\omega+\ddc\Phi)^{[n+1]}= \dd\dot{\Phi}\wedge ({\proj_1}^*\omega+\ddc\Phi)^{[n]}.
    \end{equation}
    Moreover, for any~$(1,1)$-form~$\alpha$ on~$X$
    \begin{equation}\label{eq:alpha-Wed-Om}
        \alpha \wedge (\omega+\ddc\Phi)^{[n]} =\left[ ( \alpha, \dd\dot{\varphi}_t\wedge \dc\dot{\varphi}_t )_{\varphi_t}\omega_{\varphi_t}^{[n]} + (\ddot{\varphi}_t-\abs{\dd\dot{\varphi}_t}^2_{\varphi_t}) \alpha \wedge  \omega_{\varphi_t}^{[n-1]}\right]\wedge \dd t\wedge \dd\theta.    
    \end{equation}
    We start with~\eqref{eq:ddc-APsi1}. Using~\eqref{eq:dtheta_contraction}, one can show that~\eqref{eq:ddc-APsi1} equals:
    \begin{equation}\label{eq:log_volform_d}
    \begin{split}
       & ns \int_{X\times\mathbb{D}^*} \dc\gamma\wedge\dd\left[ \log\left(\langle\chi,\mu_\Phi\rangle^{-n-1} \frac{\e^{\Psi}}{\omega^{[n]}}\right)\right]  \wedge\frac{(\omega+\ddc\Phi)^{[n]}}{\F_{s,t}^{n+1}}\\
       &+n(n+1)s^2\int_{X\times\mathbb{D}^*} \dc\gamma\wedge\log\left(\langle\chi,\mu_\Phi\rangle^{-n-1} \frac{\e^{\Psi}}{\omega^{[n]}}\right) \frac{\partial_\theta\contract(\omega+\ddc\Phi)^{[n]}}{\F_{s,t}^{n+2}}\\
       &-n(n+1)s \int_{X\times\mathbb{D}^*} \dc\gamma\wedge\log\left(\langle\chi,\mu_\Phi\rangle^{-n-1} \frac{\e^{\Psi}}{\omega^{[n]}}\right)d\langle\chi,\mu_\Phi\rangle\wedge \frac{(\omega+\ddc\Phi)^{[n]}}{\F_{s,t}^{n+2}}.
    \end{split}
    \end{equation}
    The second term in~\eqref{eq:log_volform_d} cancels out with~\eqref{eq:ddc-APsi2}, while the last term of~\eqref{eq:log_volform_d} vanishes since~$(\dc\gamma)(\chi)=0$:
    \begin{equation}\label{eq:dcgamma=0}
        \dd\langle\chi,\mu_\Phi\rangle\wedge \dc\gamma\wedge \frac{(\omega+\ddc\Phi)^{[n]}}{\F_{s,t}^{n+1}}=(\dc\gamma)(\chi)\frac{(\omega+\ddc\Phi)^{[n+1]}}{\F_{s,t}^{n+1}}=0.
    \end{equation}
    The remaining piece of~\eqref{eq:ddc-APsi1} becomes, using~\eqref{eq:dcgamma=0} and~$\ddc\log\left(\frac{e^{\Psi}}{\omega^{[n]}} \right)=\ddc\Psi-2\Ric(\omega)$,
    \begin{equation}\label{eq:lastpiece_firstintegral}
    \begin{split}
        &\int_{X\times\mathbb{D}^*} \dc\gamma\wedge d\left[ \log\left( \frac{\e^{\Psi}}{\omega^{[n]}}\right) \right]\frac{ (\omega+\ddc\Phi)^{[n]} }{\F_{s,t}^{n+1}}=\\
        =& \int_{X\times\mathbb{D}^*} \gamma\left[\ddc\Psi-2\Ric(\omega)-(n+1)\dc\log\left( \frac{\e^{\Psi}}{\omega^{[n]}}\right)\wedge \dd\log\F_{s,t} \right]\wedge\frac{ (\omega+\ddc\Phi)^{[n]} }{\F_{s,t}^{n+1}}
    \end{split}
    \end{equation}
    Next, we compute
    \begin{equation}
        \dc\log\left(\frac{e^{\Psi}}{\omega^{[n]}} \right)(\chi)=(\dc\Psi)(\chi)-\Delta_\omega\langle\chi,\mu_\omega\rangle=(\dc\Psi)(\chi)-2\langle\chi,\mu_{\Ric(\omega)}\rangle,
    \end{equation}
    so that a small modification of~\eqref{eq:dtheta_contraction} shows
    \begin{equation}\label{eq:ddc_exchange}
        \dc\log\left(\frac{e^{\Psi}}{\omega^{[n]}} \right) \wedge \dd\F_{s,t}\wedge(\omega+\ddc\Phi)^{[n]} = \left[ \dc\Psi(\chi-s\partial_\vartheta)-2\langle\chi,\mu_{\Ric(\omega)}\rangle\right]\,(\omega+\ddc\Phi)^{[n+1]}.
    \end{equation}
    We proceed by simplifying~\eqref{eq:ddc-APsi3}. Expanding the derivative, the integral on~$X$ becomes 
    \begin{equation}\label{eq:ddc_Ric}
    \begin{split}
        &\int_X \Ric(\omega)\wedge \ddc\dot{\varphi}\wedge \frac{\omega_{\varphi_t}^{[n-2]}}{f_{s,t}^n} +n\int_X  \langle \mu_{\Ric(\omega)} ,\chi\rangle\Delta_{\varphi_t}(\dot{\varphi}_t) \frac{\omega^{[n]}_{\varphi_t}}{f_{s,t}^{n+1}}\\
        &+n(n+1)\int_X  \langle \mu_{\Ric(\omega)},\chi\rangle \partial_tf_{s,t}\frac{\omega^{[n]}_{\varphi_t}}{f_{s,t}^{n+2}} -n\int_X\partial_tf_{s,t}\Ric(\omega)\wedge \frac{\omega_{\varphi_t}^{[n-1]}}{f_{s,t}^{n+1}}.
    \end{split}
    \end{equation}
    Integrating by parts the first line of~\eqref{eq:ddc_Ric} we find
    \begin{equation}
        \begin{split}
            &\int_X \Ric(\omega)\wedge \ddc\dot{\varphi}\wedge \frac{\omega_{\varphi_t}^{[n-2]}}{f_{s,t}^n} +n\int_X  \langle \mu_{\Ric(\omega)} ,\chi\rangle\Delta_{\varphi_t}(\dot{\varphi}_t) \frac{\omega^{[n]}_{\varphi_t}}{f_{s,t}^{n+1}}=\\
            =&n\int_X \left(df_{s,t}, d\dot{\varphi}_t\right)_{\varphi_t}\Ric(\omega)\wedge \frac{\omega_{\varphi_t}^{[n-1]}}{f_{s,t}^{n+1}} -n\int_X \left(\Ric(\omega), df_{s,t}\wedge \dc\dot{\varphi}\right)_{\varphi_t} \frac{\omega_{\varphi_t}^{[n]}}{f_{s,t}^{n+1}}\\
            &+n\int_X \left(d\langle \mu_{\Ric(\omega)},\chi\rangle, d\dot{\varphi}_{t}\right)_{\varphi_t} \frac{\omega^{[n]}_{\varphi_t}}{f_{s,t}^{n+1}} -n(n+1)\int_X\langle \mu_{\Ric(\omega)} ,\chi\rangle
            (\dd f_{s,t},\dd\dot{\varphi}_t)_{\varphi_t}\frac{\omega^{[n]}_{\varphi_t}}{f_{s,t}^{n+2}}.
        \end{split}
    \end{equation}
    Now note that by definition of~$f_{s,t}$,
    \begin{equation}
    \begin{split}
        (\dd f_{s,t},\dd\dot{\varphi})_{\varphi}=& (\dc\dot{\varphi})(\chi)+s\abs*{\dd\dot{\varphi}}^2_{\varphi} \\
        \left(\Ric(\omega), \dd f_{s,t}\wedge \dc\dot{\varphi}\right)_{\varphi_t}=& \left(\dd\langle\mu_{\Ric(\omega)},\chi\rangle, \dd\dot{\varphi}_{t}\right)_{\varphi_t}+s\left(\Ric(\omega), \dd\dot{\varphi}_{t}\wedge\dc\dot{\varphi}\right)_{\varphi_t}.
    \end{split}
    \end{equation}
    Substituting back in~\eqref{eq:ddc_Ric} and using~\eqref{eq:Om^n} and~\eqref{eq:alpha-Wed-Om} for~$\alpha=\Ric(\omega)$ we find that~\eqref{eq:ddc-APsi3} equals
    \begin{equation}
       2ns \int_{X\times\mathbb{D}^*}\gamma\left(\Ric(\omega)\wedge \frac{(\omega+\ddc\Phi)^{[n]}}{\F_{s,t}^{n+1}} +(n+1) \langle\chi,\mu_{\Ric(\omega)}\rangle \frac{(\omega+\ddc\Phi)^{[n+1]}}{\F_{s,t}^{n+2}}\right).
    \end{equation}
    Adding this to~\eqref{eq:lastpiece_firstintegral} and using~\eqref{eq:ddc_exchange} gives the thesis.
\end{proof}

\subsubsection{Asymptotic slopes}\label{ss:globalformulaEH}\label{sec:EH_slope}

In this Section we compute the limit of the Einstein-Hilbert functional towards the central fibre of a test configuration, proving Theorem~\ref{thm:EH_limit_testconf}. Rather than computing the limit of the total scalar curvature functional, we consider the limiting slope of the action functional, slightly generalising~\cite[\S$4.4$]{LahdiliLegendreScarpa}.

\smallskip

We fix a (smooth, ample, dominant)~$\TorusL$-equivariant test configuration~$(\tstX,\tstL)$, and a K\"ahler form~$\omega\in\chern_1(L)_+^{\TorusL}$. We can then consider the solution~$\Gamma$ to the boundary value problem~\eqref{BVP:G} and the~$\m{C}^{1,1}$-geodesic ray~$\Phi$ associated to the test configuration. Note that by~\cite[Proposition 8.8]{GZ_Book} we can extend~$\Gamma$ to an invariant function defined on~$\tstX$, and~$\Pi^*\Phi$ to an invariant function on~$\tstX\setminus\tstX_0$, which we denote by the same symbols.

\smallskip

We start by computing the limit of the volume functional, for which the argument simplifies considerably. It will be useful to consider the~$t$-primitive of the volume along the ribbon of contact forms~$\eta_{s,t}$ (see Definition~\ref{def:ribbon_geod}) as an~$\mathbb{S}^1$-invariant function on~$\mathbb{C}^*\subset\mathbb{P}^1$. Explicitly, we define
\begin{equation}
     \actvol^\chi_s(\tau)\coloneqq\int_0^t\int_X\left(\frac{(\omega+\ddc\Phi)^{[n]}}{(\langle\chi,\mu_\Phi\rangle+s\dot{\Phi})^{n+1}}\right)\dd t
\end{equation}
where~$\mu_\Phi:X\times\mathbb{D}^*\to\Pol_L$ is the moment map~$\mu_\Phi\coloneqq\mu_\omega+\dc\Phi$ of the~$\TorusX$ action.

\begin{proposition}\label{prop:AsymSlopeVol}
    The limit of the volume functional~$\Voltot(\eta_{s,\tau})$ towards the central fibre of the test configuration exists and does not depend on the choice of the initial point of the ribbon of contact forms. Moreover,
    \begin{equation}\label{AsymSlopeVol}\begin{split}
        2\pi \lim_{\tau\to 0}\Voltot(\eta_{s,\tau})
        = & (1-s\mu^{\zeta}_{\chi,\max})^{-(n+1)} \Voltot(N,\chi) -(n+1)s \Voltot(\tstN, \chi-s\zeta)\\
        = & \Voltot_s^\chi(\tstX,\tstL)
        \end{split}
    \end{equation}
    where~$\mu^{\zeta}_{\chi,\max}$ is the maximal value reached by~$\hat{\eta}^\chi(\zeta)$ on~$\tstN$, which does not depend on the choice of  Sasaki structure~$\hat{\eta}^\chi$ on~$(\tstN,\uI,\chi)$.
\end{proposition}

\begin{proof}
    We start by noting that the second variation of~$\actvol^\chi_s$ can be obtained proceeding as in the proof of Proposition~\ref{prop:ddc_A^Psi-smooth}. The direct computation shows that for any smooth~$\TorusL$-invariant subgeodesic ray~$\Phi$ we have
    \begin{equation}
        \frac{1}{(n+1)s}\ddc\actvol^{\chi}_s=-\proj_{2,\star}\left( \frac{({\proj_1}^*\omega+\ddc\Phi)^{[n+1]}}{(\langle\chi,\mu_\Phi\rangle+s\dot{\Phi})^{n+2}}\right).
    \end{equation}
    where~${\proj_1}:X\times\mathbb{C}^*\to X$ is the projection on the first factor and~$\proj_{2,\star}$ stands for the integration along the fibres of~$\proj_2\colon X\times\mathbb{C}^*\to \mathbb{P}^1$. We can pull this equation back to~$\tstX\setminus \tstX_0$ to obtain, in the weak sense of currents,
    \begin{equation}\label{eq:ddc_actvol}
        \begin{split}
        \frac{1}{(n+1)s} \ddc\actvol^{\chi}_s=-\TCmap_\star\left( \frac{(\Omega+\ddc\Gamma)^{[n+1]}}{\langle\chi-s\zeta,\mu_\Gamma\rangle^{n+2}}\right),
    \end{split}
    \end{equation}
    as~$\Pi^*(\proj_1\omega+\ddc\Phi)=\Omega+\ddc\Gamma$ away from~$\tstX_0$ and~$-\Pi^*\dot{\Phi}$ is the Hamiltonian for the~$\mathbb{S}^1$-action on~$\tstX\setminus \tstX_0$ with respect to~$\Omega_\Gamma\coloneqq\Omega+\ddc\Gamma$. Note that the form appearing on the right hand side of~\eqref{eq:ddc_actvol} is just the (pushforward of) the volume form of the CR contact form on~$\tstN$, $$\frac{\eta_\Omega+\dc\Gamma}{(\eta_\Omega+\dc\Gamma)(\xi-s\zeta)},$$ which is Sasaki with Reeb vector field~$\xi-s\zeta$.
    
    From~\eqref{eq:ddc_actvol} we obtain, integrating the second derivative of~$\actvol$,
    \begin{equation}\label{eq:intV=slop}
    \begin{split}
        (n+1)s\int_{\tstX}  \frac{(\Omega+\ddc\Gamma)^{[n+1]}}{\langle\chi-s\zeta,\mu_\Gamma\rangle^{n+2}}=& -\lim_{\tau\to 0}\int_{\mathbb{C}\setminus \mathbb{D}_{|\tau|}} \ddc\actvol^\chi_s =-2\pi\lim_{\tau\to 0}\int^{-\log|\tau|}_{-\infty} \frac{d^2 \actvol^{\chi}_s}{dt^2} dt \\ =& -2\pi\left(\lim_{t\to +\infty} \Voltot(\eta_{s,t}) - \lim_{t\to -\infty}\Voltot(\eta_{s,t})\right).    
    \end{split}
    \end{equation}
    As discussed in \S~\ref{sec:CRgeom}, by~\cite{FutakiOnoWang} the volume of a Sasaki manifold only depends on the Sasaki-Reeb vector field~$\xi-s\zeta$ (see also~\cite[\S~$1$]{FutakiMabuchi_extremal} or~\cite[Lemma 2]{Lahdili_weighted}), hence the quantity
    \begin{equation}
        \Voltot(\chi-s\zeta,\tstN) =2\pi\int_{\tstX}\frac{\Omega'^{[n+1]}}{\langle \chi-s\zeta,\mu_{\Omega'}\rangle^{n+2}}
    \end{equation}
    is independent of the choice of~$\Omega'\in[\Omega]$. In particular,~\eqref{eq:intV=slop} can be rewritten as
    \begin{equation}
        2\pi\lim_{t\to +\infty}\Voltot(\eta_{s,t}) = \lim_{t\to -\infty}2\pi\Voltot(\eta_{s,t}) -(n+1)s\Voltot(\chi-s\zeta,\tstN).
    \end{equation}
    To conclude, we can directly compute the limit as~$\tau\to+\infty$:
    \begin{equation}
        \lim_{t\to-\infty}2\pi\Voltot(\eta_{s,t})=  \lim_{\tau\to+\infty}2\pi\int_{\tstX_\tau}\frac{(\Omega+\ddc\Gamma)_{\restriction\tstX_\tau}^{[n]}}{\langle\chi-s\zeta,\mu_\Gamma\rangle^{n+1}}= \lim_{\tau\to+\infty}\int_{\tstN_\tau} \left( \frac{\eta_\Omega\wedge(\dd\eta_\Omega)^{[n]}}{\eta_\Omega(\chi-s\zeta)^{n+1}} \right)_{\restriction{\tstN_\tau}}
    \end{equation}
    where~$\eta_\Omega$ is any Sasaki structure on~$(\tstN,\uI,\xi)$ with curvature form~$\Omega$. We can write this using~$\hat{\eta}^\chi$. Indeed,~$\eta_\Omega(\chi-s\zeta)^{-1}\eta_\Omega$ is a Sasaki structure with Reeb vector field~$\chi-s\zeta$, so~$\hat{\eta}^\chi(\chi-s\zeta)^{-1}\hat{\eta}^\chi=\eta_\Omega(\chi-s\zeta)^{-1}\eta_\Omega$ as the two structures have the same kernel. Hence,
    \begin{equation}
        \lim_{t\to-\infty}2\pi\Voltot(\eta_{s,t})= \lim_{\tau\to+\infty}\int_{\tstN_\tau} \left( \frac{\hat{\eta}^\chi\wedge(\dd\hat{\eta}^\chi)^{[n]}}{(1-s\,\hat{\eta}^\chi(\zeta))^{n+1}} \right)_{\restriction{\tstN_\tau}}
    \end{equation}
    Note that the fiber~$\tstX_{\infty}$ is fixed by the test configuration action, so that~$\hat{\eta}^{\chi}(\zeta)_{\restriction\tstN_\infty}$ tends to the constant function~$\mu^\zeta_{\chi,\max}$. Hence, the limit equals
    \begin{equation}
        \lim_{t\to-\infty}2\pi\Voltot(\eta_{s,t})=\frac{1}{(1-s\mu^\zeta_{\chi,\max})^{n+1}}\int_{\tstN_\infty}\eta^\chi\wedge(\dd\eta^\chi)^{[n]}_{\restriction{\tstN_\infty}}
    \end{equation}
    which is just a multiple of the total volume of the Sasaki structure~$\eta^\chi$ on~$\tstN_\infty\simeq N$.
\end{proof}

We proceed to compute the slope of the action functional, following the same strategy as the proof of Proposition~\ref{prop:AsymSlopeVol}. However, since~$\log \omega_{\varphi}^{[n]}$, seen as a metric on~$K_{\tstX/\mathbb{P}^1}$, may blow up near the central fibre, we need to employ the twisted action functional~$\Action_\chi^\Psi$ introduced in~\eqref{eq:A_Psi-general}. Here, we choose~$\e^\Psi$ to be the~$\TorusTest$-invariant smooth Hermitian metric on~$K_{\tstX/\mathbb{P}^1}=K_{\tstX}-\TCmap^*K_{\mathbb{P}^1}$ defined by~$\Omega$ through~$\Omega^{[n+1]}\eqqcolon\e^{\Psi+\TCmap^*\log\omega_{\FS}}$, whose curvature~$2$-form satisfies
\begin{equation}\label{eq:choosing_Psi}
    -\frac{1}{2}\ddc\Psi=\Ric(\Omega)-\TCmap^*\omega_{\FS}.
\end{equation}
The asymptotic slopes as~$\tau\to 0$ of the functions~$\Action^{\Psi}_\chi$ and~$\Action_\chi$ are the same by~\cite[Lemma~$4.16$]{LahdiliLegendreScarpa}.

\begin{proposition}\label{prop:AsymCurvPSI}
    With the previous notation, the asymptotic slope of~$\Action_\chi^\Psi(s,\tau)$ near the central fibre of a test configuration satisfies
    \begin{equation}
    \begin{split}
        \lim_{t\to +\infty}\frac{d}{dt}\Action^\Psi_\chi(s,t) =&  \frac{1}{2\pi}\frac{1}{\left(1-s\mu^\zeta_{\chi,\max}\right)^n}\Scaltot(\chi,\tstN_{\infty})
        -\frac{ns}{2\pi}\Scaltot(\chi-s\zeta,\tstN)\\
        & - ns\int_{\tstX} \frac{2\TCmap^*\omega_{\FS}\wedge \Omega^{[n]} }{\langle\chi-s\zeta,\mu_\Omega\rangle^{n+1}} + n(n+1)s^2\int_{\tstX}\frac{\TCmap^*\Delta_{\omega_{\mathrm FS}}\mu_{\mathrm FS}\,\Omega^{[n+1]}}{\langle\chi-s\zeta,\mu_\Omega\rangle^{n+2}}.
    \end{split}
    \end{equation}
\end{proposition}

\begin{proof}
    We follow the proof of Proposition~\ref{prop:AsymSlopeVol}. Stokes' Theorem gives
    \begin{equation}
        \lim_{t\to+\infty}\frac{\dd}{\dd t}\Action_\chi^\Psi(s,t)=\lim_{t\to-\infty}\frac{\dd}{\dd t}\Action_\chi^\Psi(s,t)+\int_{\mathbb{P}^1}\ddc\Action_\chi^\Psi(s,t),  
    \end{equation}
    and we compute the two terms separately. We can use the expression for~$\ddc\Action_\chi^\Psi$ of Proposition~\ref{prop:ddc_A^Psi-smooth}. Recalling that~$\mu_\Gamma$ is the moment map for the~$\TorusTest$-action on~$\tstX$ with respect to~$\Omega+\ddc\Gamma$, we get
    \begin{equation}\label{eq:int_ddcActPsi}
    \begin{split}
        \frac{1}{ns}\int_{\mathbb{P}^1}\ddc \Action^\Psi_{\chi}(s,\tau) =& \int_{\tstX} \frac{\ddc\Psi\wedge (\Omega+\ddc\Gamma)^{[n]} }{\langle\chi-s\zeta,\mu_\Gamma\rangle^{n+1}}\\ 
        &-(n+1)\int_{\tstX}(\dc\Psi)(\chi-s\zeta)\frac{(\Omega+\ddc\Gamma)^{[n+1]}}{\langle\chi-s\zeta,\mu_\Gamma\rangle^{n+2}}
    \end{split}
    \end{equation}
    Now we can appeal to~\cite[Lemma 2]{Lahdili_weighted}. Taking~$\theta=\ddc\psi$ and~$\mathrm{v}(\mu')=\langle\chi-s\zeta,\mu_\Gamma\rangle^{-(n+1)}$ in the second integral~$B_{\mathrm{v}}^{\theta}$ of~\cite[Lemma 2]{Lahdili_weighted}, the right hand side of~\eqref{eq:int_ddcActPsi} does not depend on the choice of~$\Omega\in\chern_1(\tstL)_+^{\TorusTest}$, so that
    \begin{equation}\label{eq:slope_ddcActPsi}
    \begin{split}
        \frac{1}{ns}\int_{\mathbb{P}^1}\ddc \Action^\Psi_{\chi}(s,\tau) =& \int_{\tstX} \frac{\ddc\Psi\wedge \Omega^{[n]} }{\langle\chi-s\zeta,\mu_\Omega\rangle^{n+1}}-(n+1)\int_{\tstX}(\dc\Psi)(\chi-s\zeta)\frac{\Omega^{[n+1]}}{\langle\chi-s\zeta,\mu_\Omega\rangle^{n+2}}
    \end{split}
    \end{equation}
    As~$\Psi-\log(\Pi^*\omega^{[n]})=\log\left(\frac{\Omega^{[n+1]}}{\Pi^*\omega^{[n]}\wedge\TCmap^*\omega_{\mathrm FS}}\right)$ and~$\omega$ is~$\chi$-invariant,
    \begin{equation}
        \begin{split}
            \dc\Psi(\chi-s\zeta)=-\Delta_{\Omega}\langle\chi-s\zeta,\mu_\Omega\rangle-s\,\TCmap^*\Delta_{\omega_{\mathrm FS}}(\mu_{\mathrm FS}).
        \end{split}
    \end{equation}
    Plugging this identity and~\eqref{eq:choosing_Psi} in~\eqref{eq:slope_ddcActPsi} gives
    \begin{equation}\label{eq:ddc_Apsi_int}
    \begin{split}
        \frac{1}{ns}\int_{\mathbb{P}^1}\ddc \Action^\Psi_{\chi}(s,\tau) =& -\int_{\tstX} \frac{2\Ric(\Omega)\wedge \Omega^{[n]} }{\langle\chi-s\zeta,\mu_\Omega\rangle^{n+1}}+(n+1)\int_{\tstX}\frac{\Delta_{\Omega}\langle\chi-s\zeta,\mu_\Omega\rangle\,\Omega^{[n+1]}}{\langle\chi-s\zeta,\mu_\Omega\rangle^{n+2}}\\
        &- \int_{\tstX} \frac{2\TCmap^*\omega_{\FS}\wedge \Omega^{[n]} }{\langle\chi-s\zeta,\mu_\Omega\rangle^{n+1}}+s(n+1)\int_{\tstX}\frac{\TCmap^*\Delta_{\omega_{\mathrm FS}}\mu_{\mathrm FS}\,\Omega^{[n+1]}}{\langle\chi-s\zeta,\mu_\Omega\rangle^{n+2}}.
    \end{split}
    \end{equation}
    Note that if~$\eta\in\Alt^1(\tstN)$ is a Sasaki contact form such that~$\dd\eta=\Omega$, the first line on the right-hand side of~\eqref{eq:ddc_Apsi_int} is the total Tanaka-Webster scalar curvature of the Sasaki form~$\eta(\chi-s\zeta)^{-1}\eta$, c.f.\ \eqref{eq:Scaltot_Reeb}.
    
    We proceed to compute the asymptotic slope of~$\Action^{\Psi}_\chi$ as~$t\to-\infty$. The direct computation gives
    \begin{equation}\label{eq:A_psi_derivative}
    \begin{split}
        \frac{\dd}{\dd t}\Action_\chi^\Psi(s,t)= & ns\int_X \frac{1}{(\langle\chi,\mu_\varphi\rangle+s\dot{\varphi})^{n+1}} \left( \dot{\psi}_t \omega_{\varphi}^{[n]}-\dd\log\left( \frac{\e^{\psi_t}}{\omega^{[n]}}\right) \wedge \dc\dot{\varphi}\wedge\omega_{\varphi}^{[n-1]}\right) \\
        &+  \int_X \left(  \frac{2\Ric(\omega)\wedge\omega_{\varphi}^{[n-1]}}{(\langle\chi,\mu_\varphi\rangle+s\dot{\varphi})^n} -n\Delta_{\varphi}\langle\chi,\mu_\omega\rangle \frac{\omega^{[n]}_{\varphi}}{(\langle\chi,\mu_\varphi\rangle+s\dot{\varphi})^{n+1}} \right).
    \end{split}
    \end{equation}
    The~$\mathbb{C}^*$-action on~$\tstX$ is trivial near~$\tstX_{\tau=\infty}$, hence~$\e^{\psi_t}=(\rho(\tau)^*\e^\Psi)_{\restriction\tstX_1}$ tends to a constant with respect to~$t$,~$\lim_{t\to-\infty}\partial_t\e^{\psi_t}=0$. Moreover,~$\dot{\varphi}=\eta(\zeta)$ (see ~\cite[Lemma~$4.8$]{LahdiliLegendreScarpa}) tends to a constant on~$\tstX_\infty$, since~$\tstX_{\infty}$ is fixed by the~$\mathbb{C}^*$-action on~$\tstX$. Thus, the first line of~\eqref{eq:A_psi_derivative} vanishes as~$t\to-\infty$ and the limit of~$\frac{\dd}{\dd t}\Action_\chi^\Psi(s,t)$ coincides with the limit of
    \begin{equation}
        \int_{\tstX_\tau} \left(  \frac{2\Ric(\Omega)\wedge(\Omega+\ddc\Gamma)^{[n-1]}}{\langle\chi-s\zeta,\mu_\Gamma\rangle^n} -n\Delta_\Gamma\langle\chi,\mu_\Omega\rangle \frac{(\Omega+\ddc\Gamma)^{[n]}}{\langle\chi-s\zeta,\mu_\Gamma\rangle^{n+1}} \right)_{\restriction\tstX_\tau}.
    \end{equation}
    We appeal again to~\cite[Lemma 2]{Lahdili_weighted}, which shows that the above expression does not actually depend on the choice of~$\Gamma$. Hence,
    \begin{equation}
    \begin{split}
        \lim_{\tau\to+\infty} \frac{\dd}{\dd t}\Action_\chi^\Psi(s,t)  = & \lim_{\tau\to+\infty} \int_{\tstX_\tau}  \left(\frac{2\Ric(\Omega)\wedge\Omega^{[n-1]}}{\langle\chi-s\zeta,\mu_{\Omega}\rangle^n} -n\Delta\langle\chi,\mu_{\Omega}\rangle  \frac{\Omega^{[n]}}{\langle\chi-s\zeta,\mu_{\Omega}\rangle^{n+1}}\right)_{\restriction\tstX_\tau}\\
        = & \lim_{\tau\to+\infty} \frac{1}{2\pi}\Scaltot\left((\eta_\Omega(\chi-s\zeta)^{-1}\eta_\Omega)_{\restriction \tstN_\tau},\tstN_\tau\right).
    \end{split}
    \end{equation}
    As in the proof of Proposition~\ref{prop:AsymSlopeVol}, we can rewrite this in terms of~$\hat{\eta}^\chi$ as~$\hat{\eta}^\chi(\chi-s\zeta)^{-1}\hat{\eta}^\chi=\eta_\Omega(\chi-s\zeta)^{-1}\eta_\Omega$, so that
    \begin{equation}
    \begin{split}
        \lim_{t\to-\infty}\frac{\dd}{\dd t}\Action_\chi^\Psi(s,t) =& \lim_{\tau\to+\infty} \frac{1}{2\pi}\Scaltot\left((\eta_\Omega(\chi-s\zeta)^{-1}\eta_\Omega)_{\restriction \tstN_\tau},\tstN_\tau\right)\\
        = & \lim_{\tau\to+\infty} \frac{1}{2\pi}\Scaltot\left(\frac{\hat{\eta}^{\chi}}{1-s\hat{\eta}^{\chi}(\zeta)}_{\restriction \tstN_\tau},\tstN_\tau\right)=\frac{1}{2\pi}\frac{\Scaltot(\chi,\tstN_{\infty})}{(1-s\mu^\zeta_{\chi,\max})^n}.
    \end{split}
    \end{equation}
    And putting everything together, we conclude
    \begin{equation}
    \begin{split}
        \lim_{t\to+\infty}\frac{\dd}{\dd t}\Action_\chi^\Psi(s,t) =& \frac{1}{2\pi}\Scaltot(\chi-s\zeta)
        -\frac{ns}{2\pi}\Scaltot(\chi-s\zeta,\tstN)\\
         - & ns\int_{\tstX} \frac{2\TCmap^*\omega_{\FS}\wedge \Omega^{[n]} }{\langle\chi-s\zeta,\mu_\Omega\rangle^{n+1}} + n(n+1)s^2\int_{\tstX}\frac{\TCmap^*\Delta_{\omega_{\mathrm FS}}\mu_{\mathrm FS}\,\Omega^{[n+1]}}{\langle\chi-s\zeta,\mu_\Omega\rangle^{n+2}}.
        \quad \qedhere
    \end{split}
    \end{equation}
\end{proof}

\addcontentsline{toc}{section}{References}
\bibliography{Biblio25}

\end{document}